\documentclass[12pt, reqno]{amsart}
\usepackage{amsmath, amsthm, amscd, amsfonts, amssymb, graphicx, color, mathrsfs}
\usepackage[bookmarksnumbered, colorlinks, plainpages]{hyperref}
\usepackage[all]{xy}
\usepackage{slashed}
\usepackage{soul}
\usepackage{cancel}
\usepackage{slashed}
\usepackage[noabbrev, capitalise]{cleveref} 
\usepackage{tikz-cd}

\crefname{subsection}{Subsection}{Subsections}
\crefname{lemma}{Lemma}{Lemmas}
\crefname{proposition}{Proposition}{Propositions}
\crefname{corollary}{Corollary}{Corollaries}
\crefname{criterion}{Criterion}{Criteria}
\crefname{definition}{Definition}{Definitions}
\crefname{example}{Example}{Examples}
\crefname{exercise}{Exercise}{Exercises}
\crefname{conclusion}{Conclusion}{Conclusions}
\crefname{conjecture}{Conjecture}{Conjectures}
\crefname{summary}{Summary}{Summaries}
\crefname{axiom}{Axiom}{Axioms}
\crefname{problem}{Problem}{Problems}
\crefname{remark}{Remark}{Remarks}

\newcommand{\B}{\mathbb{B}}
\newcommand{\C}{\mathbb{C}}
\newcommand{\bH}{H}

\newcommand{\N}{\mathbb{N}}

\newcommand{\R}{\mathbb{R}}

\newcommand{\fB}{\mathfrak{B}}
\newcommand{\fD}{\mathfrak{D}}
\newcommand{\fM}{\mathfrak{M}}

\newcommand{\dst}{\displaystyle}

\newcommand{\scal}{\mbox{\bf{(}}}
\newcommand{\scar}{\mbox{\bf{)}}}
\newcommand{\dG}{\overset{\text{\tiny$\bullet$}}{G}}
\newcommand{\cH}{{\mathcal H}}
\newcommand{\cS}{{\mathcal S}}

\DeclareMathOperator*{\esssup}{ess\,sup}

\newcommand{\group}{G}

\newcommand{\abs}[1]{\left| #1 \right|}
\newcommand{\conj}[1]{\overline{#1}}

\newcommand{\pDiffToOrd}[2]{\partial_{#1}^{#2}}
\newcommand{\pDiffGrpTo}[1]{\mathfrak{D}_{#1}}
\newcommand{\pDiffGrpToOrd}[2]{\pDiffGrpTo{#1}^{#2}}

\newcommand{\CinftyOn}[1]{C^{\infty}(#1)}
\newcommand{\CinftyCSOn}[1]{C^{\infty}_{c}(#1)}

\newcommand{\SchwartzSpaceOn}[1]{\mathcal{S}(#1)}

\newcommand{\LpnormOver}[3]{\lVert #3 \rVert_{L^{#1}(#2)}}

\newcommand{\inProdIn}[3]{\left\langle #1, #2 \right\rangle_{#3}}

\newcommand{\pull}[1]{#1_{*}}

\newcommand{\FTEucl}{\mathscr{F}_{\R^{n}}}
\newcommand{\FTGrp}{\mathscr{F}_{\group}}

\newcommand{\HaarMeas}{\mu_{\group}}

\newcommand{\symbClassOn}[4]{S^{#1}_{#2,#3}(#4)}
\newcommand{\PsiDOHor}[4]{\Psi^{#1}_{#2,#3}(#4)}

\newcommand{\opRnOf}[1]{#1(t,D)}

\newcommand{\symbToGrp}[1]{#1^{\group}}
\newcommand{\symbToRn}[1]{#1^{\R^{n}}}

\newcommand{\adj}[1]{#1^{*}}

\newtheorem{theorem}{Theorem}[section]
\newtheorem{lemma}[theorem]{Lemma}

\newtheorem{proposition}[theorem]{Proposition}
\newtheorem{corollary}[theorem]{Corollary}

\theoremstyle{definition}
\newtheorem{definition}[theorem]{Definition}

\theoremstyle{remark}
\newtheorem{remark}[theorem]{Remark}
\numberwithin{equation}{section}

\begin{document}
\setcounter{page}{1}

\title[Global pseudo-differential operators on $G=(-1,1)^n$]{Global pseudo-differential operators on the Lie group $G= (-1,1)^n$}

\author[D. Cardona]{Duv\'an Cardona}
\address{
  Duv\'an Cardona:
  \endgraf
  Department of Mathematics: Analysis, Logic and Discrete Mathematics
  \endgraf
  Ghent University, Belgium
  \endgraf
  {\it E-mail address} {\rm duvanc306@gmail.com, duvan.cardonasanchez@ugent.be}
  }

\author[R. Duduchava]{Roland Duduchava }
\address{
  Roland Duduchava:
  \endgraf
  Institute of Mathematics
  \endgraf
  The University of Georgia
  \endgraf
  Tbilisi, Georgia
   \endgraf
 and
  \endgraf
  A. Razmadze Mathematical Institute
  \endgraf
  Tbilisi State University
  \endgraf
  Georgia
  \endgraf
  {\it E-mail address} {\rm roldud@gmail.com, r.duduchava@ug.edu.ge}
  }

\author[A. Hendrickx]{Arne Hendrickx }
\address{
 Arne Hendrickx
  \endgraf
  Department of Mathematics: Analysis, Logic and Discrete Mathematics
  \endgraf
  Ghent University, Belgium
  \endgraf
  {\it E-mail address:} {\rm arnhendr.Hendrickx@UGent.be}
  }

\author[M. Ruzhansky]{Michael Ruzhansky}
\address{
  Michael Ruzhansky:
  \endgraf
  Department of Mathematics: Analysis, Logic and Discrete Mathematics
  \endgraf
  Ghent University, Belgium
  \endgraf
 and
  \endgraf
  School of Mathematical Sciences
  \endgraf
  Queen Mary University of London
  \endgraf
  United Kingdom
  \endgraf
  {\it E-mail address} {\rm michael.ruzhansky@ugent.be, m.ruzhansky@qmul.ac.uk}
  }

\thanks{The authors are supported  by the FWO  Odysseus  1  grant  G.0H94.18N:  Analysis  and  Partial Differential Equations and by the Methusalem programme of the Ghent University Special Research Fund (BOF)
(Grant number 01M01021). Michael Ruzhansky is also supported  by EPSRC grant
EP/R003025/2. 
R. Duduchava is supported by the grant of the Shota Rustaveli Georgian National Science Foundation FR-19-676.
}


\begin{abstract}  In this work we characterise the H\"ormander classes
$\symbClassOn{m}{\rho}{\delta}{\group,\textnormal{H\"or}}$ on the open
manifold $\group = (-1,1)^n$. We show that by endowing the open manifold
$\group = (-1,1)^n$ with a group structure, the corresponding global Fourier
analysis  on the group allows one to define a global notion of symbol on the
phase space $\group \times \R^n$. Then, the class of pseudo-differential
operators associated to the global H\"ormander classes
$\symbClassOn{m}{\rho}{\delta}{\group \times \R^n}$ recovers the H\"ormander
classes $\symbClassOn{m}{\rho}{\delta}{\group,\textnormal{loc}}$ defined by
local coordinate systems. The analytic and qualitative properties of the
classes $\symbClassOn{m}{\rho}{\delta}{\group \times \R^n}$ are presented
in terms of the corresponding global symbols. In particular, $L^p$-Fefferman
type estimates and Calder\'on-Vaillancourt theorems are analysed, as well as the spectral properties of the operators.
\end{abstract}

\maketitle

\tableofcontents
\allowdisplaybreaks

\section{\bf Introduction}
The open set $G=(-1,1)^n$ endowed with the operation $x+_{G}y=(x+y)/(1+x\cdot y)$ is a non-compact Lie group and in terms of the Fourier analysis associated to the group $(G,+_G),$  in this work we characterise the H\"ormander classes $\Psi^{m}_{\rho,\delta}(G,\textnormal{H\"or})$ of pseudo-differential operators on $G$ defined by local coordinate systems. The approach described here can be extended for instance, to any star-shaped open sub-set $\tilde{G}$ of $\R^n$ in view if the natural diffeomorphism $G\cong \tilde{G}$.  Indeed, in terms of the Fourier transform $\mathscr{F}_G$ on $G,$ to any continuous linear operator $A$ on $C^{\infty}(G)$ we associate a global distribution $\sigma_A $ on $G\times \R^{n}$ in such a way that the quantisation of the symbol $\sigma_A$ gives the operator according to the formula
\begin{equation}
\forall f\in C^{\infty}(G),\,\,    Af=\mathscr{F}_{G}^{-1}[\sigma_A(x,\xi)[\mathscr{F}_{G}f]].
\end{equation}Then, when the distribution $\sigma_A$ agrees with a function on the phase space $G\times \R^n$, we analyse the properties of the operator $A$ in terms of the properties of the symbol $\sigma_A.$
Summarising the results of this manuscript, we have investigated:
\begin{itemize}
    \item Asymptotic expansions for the composition, the adjoint  and the parametrices (of elliptic operators) for the global H\"ormander classes $\Psi^{m}_{\rho,\delta}(G\times \R^n).$
    \item The mapping properties of the classes $\Psi^{m}_{\rho,\delta}(G\times \R^n)$ on $L^p$-spaces on $G.$ With $p\neq 2$ we prove a $L^p$-Fefferman type theorem, and with $p=2$ we obtain the Calder\'on-Vaillancourt for these classes.
    \item We prove the corresponding Gohberg lemma for a suitable sub-class of the family  $\Psi^{0}_{\rho,\delta}(G\times \R^n)$ and the chracterisation of compact operators on $L^2(G)$ is established.
    \item We prove the Atiyah-Singer-Fedosov index formula in our setting. Indeed, we prove for the Shubin class of elliptic operators of order zero the index formula:
    \begin{equation}\label{Index:A:Intro}
     \textnormal{ind}[A]=-\frac{(n-1)!}{(-2\pi i)^{n}(2n-1)!}\int\limits_{\partial{B}}\textnormal{Tr}[a^{-1}(x,\xi)da(x,\xi)]^{2n-1}.
 \end{equation}The left-hand side in \ref{Index:A:Intro} is the Fredholm index and the right hand side is the ``winding number" of $a.$
 \item Other spectral properties for the pseudo-differential calculus on $G=(-1,1).$
\end{itemize}

In the case of the torus $\mathbb{T}^n=[-1,1]^n,$ $1\sim -1,$ the global characterisation of the H\"ormander classes $\Psi^{m}_{\rho,\delta}(\mathbb{T}^n,\textnormal{H\"or})$ was done by MacLean in \cite{Mc}. Also, an alternative proof for this fact was done in \cite{Ruz} using a  periodisation  technique compatible with a global notion of symbol on the phase space $\mathbb{T}^n\times \mathbb{Z}^n,$ (instead of the phase space  $G\times \R^n=(-1,1)^n\times \R^n$). For the  spectral and the analytical properties (and their applications) of the pseudo-differential calculus on the torus we refer the reader to  \cite{ag},  \cite{Duvan2}, \cite{Duvan3}, \cite{Duvan4}, \cite{KumarCardona2}, \cite{Ournote}, \cite{CRS2018},  \cite{Profe},  \cite{s1}, \cite{m}, \cite{tur}, \cite{Ruz-2}, and, mainly, the  reference \cite{Ruz}.

The construction of pseudo-differential operators using the Lie group approach as in this work is parallel to the pseudo-differential theories in \cite{Ruz}, \cite{FischerRuzhanskyBook} where a global notion of symbol on the phase space $\mathbb{G}\times \widehat{\mathbb{G}}$  has been consistently developed, with $\mathbb{G}$ being a Lie group with a good Fourier analysis induced by its unitary dual $\widehat{\mathbb{G}}$. Even, generalising the global quantisation  from  the torus  \cite{Ruz-2} to any compact Lie group as well as their applications many results were derived in the last years. Indeed, the applications of the global quantisation on compact Lie groups, its analytical and spectral properties as well as their applications for the analysis of PDE, index theorems, regularisation of traces and other aspects of the geometric and harmonic analysis can be found e.g. in  \cite{Cardona5}, \cite{Cardona22}, \cite{Cardona6}, \cite{CdC2}, \cite{CardonaRuzhansky2017I}, \cite{CardonaRuzhansky2019I}, \cite{CardonaRuzhanskyCollectanea}, \cite{Subelliptic:calculus},  \cite{CardonaRuzhanskyMZ}, \cite{CDR21b}, \cite{CDRJevEq},  \cite{CDRMonaMath}, \cite{RuzhanskyDelgado2017}, \cite{DelRuzTrace1}, \cite{DelRuzTrace11}, \cite{DelRuzTrace111}, \cite{DelRuzTrace1111}, \cite{deMoraes}, \cite{RuzhanskyDelgadoCardona2019},  \cite{ARLG}, \cite{GarettoRuzhansky2015},  \cite{NurRuzTikhBesov2015}, \cite{NurRuzTikhBesov2017},     \cite{M1}, \cite{M2},  \cite{RodriguezRuzhansky2020}, \cite{Ruz}, \cite{Ruz-Tok}, \cite{ProfRuzM:TokN:20017}, \cite{RuzhanskyTurunenWirth2014}, \cite{RuzTurIMRN}, \cite{RuzhanskyWirth2014}, \cite{RuzhanskyWirth2015} and in the extensive list of references of these works.

On the other hand, the Fourier transform $\mathcal{F}_G$ and the convolution of functions $\varphi\ast_G\psi$ on the interval $G=(-1,1)$ (that corresponds to the one dimensional case; see below) were firstly  defined by Petrov in \cite{Pe06a,Pe06b} by using the diffeomorphism $x\;:\;\R\to G$ and its inverse $t\;:\;G\to\R$. In these papers and in \cite{SP20}, the defined Fourier transform $\mathcal{F}_G$ and convolution $\varphi\ast_G\psi$ were used for the investigation of convolution and differential equations, such as Prandtl, Tricomi, Lavrentjev-Bitsadze integral and integro-differential equations, Laplace-Beltrami equation on the sphere and some other equations from the Mathematical Physics. The  authors in \cite{Pe06a,Pe06b,SP20} did not utilise  the group structure of $G=(-1,1)$. However, the group structure of $G=(-1,1)$ was used in \cite{Du22} for the intrinsic definition of the Haar measure $d_G x$, the Fourier transform $\mathcal{F}_G$ and the Fuchs-type differential operator $\mathfrak{D}$ (see below). These tools enabled one to define in \cite{Du22} the Bessel potential-type spaces, to prove theorems on multipliers for convolution operators and derive more precise results for the above mentioned convolution equations and other applications to similar models arising from the mathematical physics. 

This work will be dedicated to the consistent development of the global pseudo-differential calculus on $G=(0,1).$ In particular, the H\"ormander classes on $G$ as defined in \cite{Hormander1985III}, will be characterised by this approach. This work is organised as follows:
\begin{itemize}
    \item In Section \ref{Pre} we present the basics on the H\"ormander pseudo-differential calculus, we also  present the topics related with the Fourier analysis on the group $G,$ and the definition of the $L^p$-Sobolev spaces in this setting. 
    \item In Section  \ref{HClassesG} we introduce the global classes of pseudo-differential operators on $G,$ the corresponding  quantisation formula and then we prove that these classes are stable under compositions, adjoints, and the  construction of parametrices. 
    \item In Section \ref{mapping:properties} we study the mapping properties of these classes, we establish the Calder\'on-Vaillancourt theorem, the G\r{a}rding inequality and the Fefferman-Phong inequality. 
    \item In Section \ref{Spectral:1}  we investigate the spectral properties of the pseudo-differential calculus and in particular, the Atiyah-Singer-Fedosov theorem is proved in this setting. 
    \item Finally in Section \ref{Spectral:2} the Fredholmeness of the pseudo-differential operators is analysed as well as other mapping properties on $L^p$-Sobolev spaces. In this last section the results are derived  from their corresponding analogues for Fourier multipliers.
\end{itemize}

\section{\bf Preliminaries}\label{Pre}
In this section we provide the preliminaries used in this work about the theory of pseudo-differential operators. This theory will be introduced for open subsets of the Euclidean space.  To do this we follow H\"ormander \cite{Hormander1985III}. The Lie structure of $G=(-1,1)^n$ and its Fourier analysis  are discussed here as well as, the function spaces on $G$ of interest for this work are defined. For the aspects about the Lie theory we follow \cite{Ruz}.

\subsection{H\"ormander classes on open sets of \texorpdfstring{$\R^n$}{Rn}}

Let us  introduce the H\"ormander classes starting with the definition in the Euclidean setting. 
\begin{definition}[Pseudo-differential operators on Euclidean open sets]
Let $U$ be an open  subset of $\R^n$ such that $U \neq \emptyset$ and $U \neq \R^n.$ We say that the ``symbol" $a \in \CinftyOn{U \times \R^n, \C}$ belongs to the H\"ormander class of order $m$ and of $(\rho,\delta)$-type, denoted by
$  \symbClassOn{m}{\rho}{\delta}{U,\textnormal{H\"or}},$  
where $0\leqslant \rho,\delta \leqslant 1,$ if for every compact subset $K \subset U$ and for all $\alpha,\beta \in \N_0^n$, the symbol inequalities
\begin{equation}\label{Hormander:classes:U}
  \abs{\pDiffToOrd{x}{\beta} \pDiffToOrd{\xi}{\alpha} a(x,\xi)} \leqslant C_{\alpha,\beta,K}(1+\abs{\xi})^{m-\rho\abs{\alpha}+\delta\abs{\beta}},
\end{equation} hold true uniformly in $x \in K$ for all $\xi\in \R^n$. Then, a continuous linear operator $A : \CinftyCSOn{U} \rightarrow \CinftyOn{U}$
is a pseudo-differential operator of order $m$ of $(\rho,\delta)$-type, if there exists a symbol $a \in S^{m}_{{\rho},{\delta}}{(U,\textnormal{H\"or})}$ such that
\begin{equation*}
 \forall x\in \mathbb{R}^n,\,\,   Af(x) = \int\limits_{\R^{n}}{e^{2\pi i x \cdot \xi} \, a(x,\xi) \, (\FTEucl{f})(\xi)}d{\xi},
\end{equation*} for all $f \in \CinftyCSOn{U},$ where
$$
 \forall\xi\in \mathbb{R}^n,\,\,   (\FTEucl{f})(\xi) := \int\limits_{\mathbb{R}^n}{e^{-2 \pi i x\cdot \xi} \, f(x)}d{x}
$$ is the Euclidean Fourier transform of $f$ at $\xi \in \R^n,$ and $U$ is identified with $\mathbb{R}^n$.
\end{definition}
\begin{remark}[Pseudo-differential operators on $\R^n$] In the specific case where $U$ is the whole space $\R^n,$ the conditions of the uniform estimates in \eqref{Hormander:classes:U} on compact subsets on $\R^n$ are removed and the following symbol estimates
\begin{equation}\label{Hormander:classes:U2}
  \abs{\pDiffToOrd{x}{\beta} \pDiffToOrd{\xi}{\alpha} a(x,\xi)} \leqslant C_{\alpha,\beta}(1+\abs{\xi})^{m-\rho\abs{\alpha}+\delta\abs{\beta}},
\end{equation} for functions $a \in \CinftyOn{\R^{n} \times \R^{n}}$ define the symbol class $\symbClassOn{m}{\rho}{\delta}{\R^n \times \R^n} := \symbClassOn{m}{\rho}{\delta}{\R^n;\textnormal{H\"or}}$.
\end{remark}

\subsection{The group \texorpdfstring{\(\group\)}{group}}

Here we explain how we can endow \((-1,1)^{n}\) with an intrinsic group structure. The general idea is that if we have a bijection from a group to a set, we can force this bijection to become an isomorphism by defining a group operation on the set by this condition.
\begin{remark}\label{Remark:1}
Suppose \((G,\cdot)\) is an arbitrary group, \(H\) is a set and \(\theta : G \to H\) is a bijection. Then we can endow \(H\) with a group structure such that \(\theta\) becomes a group isomorphism between \(G\) and \(H\). Since \(\theta\) is a bijection, we can write every element \(h \in H\) as \(h = \theta(g)\) for some \(g \in G\). We define the group operation \(\cdot_{H}\) on \(H\) by
\[\theta(g_{1}) \cdot_{H} \theta(g_{2}) := \theta(g_{1} \cdot g_{2}) \in H.\]
The neutral element in \(H\) is \(1_{H} := \theta(1)\), and one easily finds that
\[\theta(g) \cdot_{H} \theta(1) = \theta(g) = \theta(1) \cdot_{H} \theta(g)\]
and
\[\theta(g) \cdot_{H} \theta(g^{-1}) = \theta(1) = \theta(g^{-1}) \cdot_{H} \theta(g).\]
Similarly, the associativity of \(\cdot_{H}\) follows from the associativity of \(\cdot\) in \(G\). We thus find indeed that \((H,\cdot_{H})\) is a group and \(\theta\) is an isomorphism by the definition of \(\cdot_{H}\).
\end{remark}
\begin{remark}[The Lie group structure on $G=(-1,1)^n$]
We now apply Remark \ref{Remark:1} to the particular case of the bijection
\begin{equation}\label{x:mapping}
    x : \R^{n} \to (-1,1)^{n} : t \mapsto x(t) := (-\tanh{t_{1}}, \cdots, -\tanh{t_{n}}),
\end{equation} where
\begin{equation*}
   \tanh(t):=\frac{\sinh(t)}{\cosh(t)}=\frac{e^{t}-e^{-t} }{e^{t}+e^{-t}},\,t\in \mathbb{R}. 
\end{equation*}
Note that the minus sign is only a matter of convention. We thus endow \(\group = (-1,1)^{n}\) with the group operation
\[(-\tanh{t_{1}}) +_{\group} (-\tanh{t_{2}}) := -\tanh(t_{1} + t_{2}) = \frac{(-\tanh{t_{1}}) + (-\tanh{t_{2}})}{1 + (-\tanh{t_{1}})(-\tanh{t_{2}})}.\]
Since we dispose of an \emph{addition formula} for the hyperbolic tangent function, we can write the group operation more neatly by setting \(x := -\tanh{t_{1}}\) and \(y := -\tanh{t_{2}}\) so that we get
\[x +_{\group} y = \frac{x+y}{1+xy}.\]
This endows \(\group = (-1,1)^{n}\) with a Lie group structure such that \(x\) is an isomorphism, whose inverse is given by
\begin{equation}\label{t:mapping}
    t : (-1,1)^{n} \to \R^{n} : x \mapsto t(x) := \left(\frac{1}{2} \ln\left(\frac{1-x_{1}}{1+x_{1}}\right), \dots, \frac{1}{2} \ln\left(\frac{1-x_{n}}{1+x_{n}}\right)\right).
\end{equation}
These isomorphisms \(x\) and \(t\) will play a very important role in the rest of this paper as they provide a canonical translation of properties of the Lie group \(\R^{n}\) to the corresponding properties on \(\group = (-1,1)^{n}\). Then, $G$ has now a natural structure of a Lie group.
\end{remark}

\subsection{Function spaces on \texorpdfstring{$\group = (-1,1)^n$}{group}}\label{subsection:funcSpaces}
We continue the idea that the isomorphisms \(x\) and \(t\) translate properties and concepts of \(\R^n\) to \(\group\) by constructing function spaces on \(\group\). To this end it will be convenient to consider the pull-backs \(\pull{t}\) and \(\pull{x}\).

\begin{definition}[The pull-backs  $t_*$ and $x_*$]
The pull-back $\pull{t}f$ of a (measurable) function $f : \R^n \to \C$, with $t$ as in \eqref{t:mapping}, is defined via
\begin{equation}
    \pull{t}f := f\circ t: \group \to \C.
\end{equation}
Similarly, if \(f : \group \to \C\) then \(\pull{x}f(t) := f(x(t))\). 
\end{definition}

\begin{remark}
These pull-backs will allow one to switch between function spaces on \(\R^{n}\) and \(\group\) as we will show now.

If \(f \in \CinftyOn{\R^n}\) then \(\pull{t}f : \group = (-1,1)^n \to \C\) is a smooth function as a composition of smooth functions. Conversely, for a smooth function \(f : \group \to \C\) we have that \(\pull{x} f\) is a smooth function on \(\R^n\). Since \(\pull{t}\) and \(\pull{x}\) are inverse bijections, this establishes a canonical bijective correspondence given by \(\CinftyOn{\group} = \pull{t}(\CinftyOn{\R^n})\).
\end{remark}

\begin{definition}[The natural distance on $G$]
Define the distance \(d_{\group} : \group \times \group \to [0,\infty)\) by
\(d_{\group}(x,y) := \abs{t(x) - t(y)}\). Note that this definition is chosen in a such a way that \(t : G \to \R^{n}\) and \(x : \R^{n} \to G\) become isometries. As a consequence, this endows \(G\) with a topology, which is homeomorphic with the Euclidean topology. For this topology it can easily be checked that we have a bijective correspondence between the spaces of compactly supported smooth functions on \(\R^n\) and \(\group\) given by \(\CinftyCSOn{\group} = \pull{t}(\CinftyCSOn{\R^{n}})\).
\end{definition}
\begin{definition}[Canonical vector fields on $G$]
Consider \(f \in \CinftyOn{\R^n}\) and \(g \in \CinftyOn{\group}\). We use the expressions \eqref{t:mapping} and \eqref{x:mapping} to compute that
\[\frac{\partial t_{k}}{\partial x_{j}}(x) = -\delta_{jk} \frac{1}{(1-x_{k}^{2})} \qquad \textnormal{and} \qquad \frac{\partial x_{k}}{\partial t_{j}}(t) = -\delta_{jk} (1-x_{k}^{2}),\]
where \(\delta_{jk}\) is the Kroneker delta, which is \(1\) if \(j=k\) and \(0\) otherwise.
A change of variables then leads to the formulas
\begin{equation}\label{eq:cov1}
    \frac{\partial (\pull{t}f)}{\partial x_{j}}(x) = \sum_{k=1}^{n} \frac{\partial f}{\partial t_{k}}(t(x)) \frac{\partial t_{k}}{\partial x_{j}}(x) = - \frac{1}{(1-x_{j}^{2})} \frac{\partial f}{\partial t_{j}}(t(x))
\end{equation}
and
\begin{equation}\label{eq:cov2}
    \frac{\partial (\pull{x}g)}{\partial t_{j}}(t) = \sum_{k=1}^{n} \frac{\partial g}{\partial x_{k}}(x(t)) \frac{\partial x_{k}}{\partial t_{j}}(t) = - (1-x_{j}^{2}) \frac{\partial g}{\partial x_{j}}(x(t)).
\end{equation}
These equations motivate the introduction of the partial differential operators
\[\pDiffGrpTo{x_{j}} := -(1-x_{j}^{2}) \frac{\partial}{\partial x_{j}}\]
for \(1 \leqslant j \leqslant n\). Note that \eqref{eq:cov1} can be rewritten as
\[\pDiffGrpTo{x_{j}} (\pull{t}f)(x) = \frac{\partial f}{\partial t_{j}}(t(x)) = \pull{t} \frac{\partial f}{\partial t_{j}}(x).\]
Since this identity is true for all \(x \in G\) and all \(f \in \CinftyOn{\R^{n}}\), we have the equality of operators
\begin{equation}\label{eq:derivativeTransform1}
    \pull{t} \frac{\partial}{\partial t_{j}} = \pDiffGrpTo{x_{j}} \pull{t}.
\end{equation}
This equality shows that the operators \(\pDiffGrpTo{x_{j}}\) for \(1 \leq j \leq n\) are the natural partial differential operators on \(G\) since \(\pull{t}\) in a sense transforms \(\frac{\partial}{\partial t_{j}}\) to \(\pDiffGrpTo{x_{j}}\). Similarly, it follows from \eqref{eq:cov2} that
\begin{equation}\label{eq:derivativeTransform2}
    \pull{x} \pDiffGrpTo{x_{j}} = \frac{\partial}{\partial t_{j}} \pull{x}.
\end{equation}
\end{definition}
\begin{definition}[Canonical partial differential operators on $G$] Before we proceed, let us introduce some common notation. For multi-indices \(\alpha, \beta \in \N_{0}^{n}\) we write
\[\pDiffToOrd{t}{\alpha} := \pDiffToOrd{t_{1}}{\alpha_{1}} \pDiffToOrd{t_{2}}{\alpha_{2}} \cdots \pDiffToOrd{t_{n}}{\alpha_{n}} \qquad \textnormal{and} \qquad \pDiffGrpToOrd{x}{\alpha} := \pDiffGrpToOrd{x_{1}}{\alpha_{1}} \pDiffGrpToOrd{x_{2}}{\alpha_{2}} \cdots \pDiffGrpToOrd{x_{n}}{\alpha_{n}}.\]
We denote the length of the multi-index \(\alpha \in \N_{0}^{n}\) by
\[\abs{\alpha} := \alpha_{1} + \alpha_{2} + \cdots + \alpha_{n}.\]
\end{definition}


Now we introduce the space of Schwartz functions on \(\group\).
\begin{remark}[The Schwartz class on $G$]
Let \(f \in \SchwartzSpaceOn{\R^{n}}\). Then for every two multi-indices \(\alpha, \beta \in \N_{0}^{n}\), we consider the seminorm
\[\sup_{t \in \R^{n}} \abs{t^{\alpha} \pDiffToOrd{t}{\beta}f(t)} < \infty.\]
Observe that we can rewrite this seminorm using \eqref{eq:derivativeTransform1} as
\[\sup_{t \in \R^{n}} \abs{t^{\alpha} \pDiffToOrd{t}{\beta}f(t)} = \sup_{x \in \group} \abs{t(x)^{\alpha} \pDiffToOrd{t}{\beta}f(t(x))} = \left(\frac{1}{2}\right)^{\abs{\alpha}} \sup_{x \in \group} \abs{\left(\ln \frac{1-x}{1+x}\right)^{\alpha} \pDiffGrpToOrd{x}{\beta}\pull{t}f(x)}.\]
The latter expression can thus be used for the seminorms defining the Schwartz space \(\SchwartzSpaceOn{G}\).

Note that \(\ln\frac{1-y}{1+y}\) is asymptotically equivalent to \(\ln(1-y^{2})\) on \((-1,1)\) (with the topology induced by \(-\tanh\)) because
\[\frac{\ln\frac{1-y}{1+y}}{\ln(1-y^{2})} = 1 - \frac{\ln(1+y^{2})}{\ln(1-y^{2})} \xrightarrow{y \to \pm 1} 1.\]
Hence, we may replace the factor \(\left(\ln\frac{1-x}{1+x}\right)^{\alpha}\) in the expression for the seminorms by \({[\ln(1-x^{2})]^{\alpha}}\). 
We thus define the Schwartz space on \(\group\) to be the space of all smooth functions \(f\) on \(G\) such that the seminorms
\[\sup_{x \in \group} \abs{{[\ln(1-x^2)]^{\alpha}} \pDiffGrpToOrd{x}{\beta}f(x)} < \infty.\]
Clearly, we also have that \(\SchwartzSpaceOn{G} = \pull{t}(\SchwartzSpaceOn{\R^n})\), as we expected.
\end{remark}

\begin{remark}[Lebesgue spaces on $G$]
Finally, we construct the Lebesgue spaces on \(G\). Let \(f \in L^{p}(\R^{n})\) for \(0 < p \leqslant \infty\). If \(p < \infty\), we find using a change of variables that
{\[\LpnormOver{p}{\R^{n}}{f}^{p} = \int\limits_{\R^{n}}{\abs{f(t)}^{p}}dt= \int\limits_{(-1,1)^{n}}\frac{\abs{f(t(x))}^{p}\mathrm{d}x}{(1-x_{1}^{2}) \cdots (1-x_{n})^{2}}.\]}
This shows that a natural measure on \(G = (-1,1)^n\) is given by
\[\mathrm{d}\HaarMeas(x) := \frac{\mathrm{d}x}{(1-x_1^2)(1-x_2^2) \cdots (1-x_n^2)},\]
where \(\mathrm{d}x\) is the Lebesgue measure on \(\R^n\). This indeed happens to be the Haar measure on \(G\). 
Thus, \(L^{p}(G) = \pull{t}(L^{p}(\R^n))\), so that \(\pull{t}\) is an \(L^{p}\)-isometry, and the corresponding \(L^{p}\)-norm of an element \(f \in L^{p}(G)\) is given by
\[\LpnormOver{p}{G}{f} = \left(\int\limits_{G}
     {\abs{f(x)}^{p}}{{d\mu_{G}(x)}}
     \right)^{\frac{1}{p}}.\]
Moreover, since \(\pull{t}\) is an \(L^{2}\)-isometry (or by a direct computation) we find for all \(f,g \in L^{2}(\R^n)\) that
\begin{equation}\label{eq:innerProducts}
    \inProdIn{f}{g}{L^{2}(\R^n)} = \int\limits_{\R^{n}}{f(t) \, \conj{g(t)}}{{dt}} = \int\limits_{\group}{f(t(x)) \, \conj{g(t(x))}}{d\HaarMeas(x)} = \inProdIn{\pull{t}f}{\pull{t}g}{L^{2}(\group)}.
\end{equation}
For the case \(p = \infty\) we similarly find that \(L^{\infty}(G) = \pull{t}(L^{\infty}(\R^n))\) with norm
\[\LpnormOver{\infty}{G}{f} = \esssup_{x \in G} \abs{f(x)}.\]
Note that the inner product on $L^2(G)$ is given by
\begin{equation}
    (u,v)_{L^2(G)}=\int\limits_{(-1,1)^n}\frac{u(x)\overline{v(x)}\mathrm{d}x}{(1-x_1^2)(1-x_2^2) \cdots (1-x_n^2)}.
\end{equation}
\end{remark}

\subsection{Global Fourier analysis on \texorpdfstring{$\group = (-1,1)^n$}{group=(-1,1)n}} The Euclidean Fourier transform induces  the  group Fourier transform $\mathcal{F}_{G}$ on $G.$ We observe it in the following remark.
\begin{remark}\label{FT:Remark}
Let \(f \in \SchwartzSpaceOn{\group}\). Then \(\pull{x}f \in \SchwartzSpaceOn{\R^{n}}\) and we have

\begin{align*}
&\FTEucl[\pull{x}f](\xi)\\
&= \int\limits_{\R^{n}}{e^{-2\pi i t \cdot \xi} \, \pull{x}f(t)}{dt}
     =\int\limits_{(-1,1)^n} \frac{e^{-2\pi it(x)\cdot\xi} \, f(x) \, \mathrm{d}x}{(1-x_{1}^{2}) \cdots (1-x_{n}^{2})}=\int\limits_{(-1,1)^n}e^{-2 \pi it(x)\cdot\xi}\,f(x)\, \mathrm{d}\mu_G(x)\\
     &=:(\mathscr{F}_{G}f)(2\xi),\,\xi\in \mathbb{R},
\end{align*}
where we applied the change of variables \(x = x(t)\). We thus find that
\begin{equation}\label{eq:FTToGrp}
    \FTEucl[\pull{x}f]\left(\frac{\xi}{2}\right) = \int\limits_{\group}{\prod_{j=1}^{n} \left(\frac{1-x_{j}}{1+x_{j}}\right)^{-i \pi \xi_{j}} f(x)}{d\HaarMeas(x)} = \int\limits_{\group}{\left(\frac{1-x}{1+x}\right)^{-i \pi \xi} f(x)}{d\HaarMeas(x)},
\end{equation}
with the multilinear notation
\[\left(\frac{1-x}{1+x}\right)^{-i \pi \xi} := \left(\frac{1-x_{1}}{1+x_{1}}\right)^{-i \pi \xi_{1}} \cdots \left(\frac{1-x_{n}}{1+x_{n}}\right)^{-i \pi \xi_{n}}.\]
\end{remark}
\begin{definition}[Group Fourier transform]
In view of Remark \ref{FT:Remark},  we define the Fourier transform on \(\group\) of a function \(f \in \SchwartzSpaceOn{\group}\) by
\begin{equation}
   \FTGrp[f](\xi) := \int\limits_{\group}{\left(\frac{1-x}{1+x}\right)^{-i\pi \xi} f(x)}{d\HaarMeas(x)},\qquad {\xi\in\R^n}.
\end{equation}
We can provide similar motivation for defining the inverse Fourier transform as (\cite{Du22})
\begin{equation}\label{FIF}
    \FTGrp^{-1}[f](x):=
     \int\limits_{\R^n}{\left(\frac{1-\xi}{1+\xi}\right)^{i\pi x}
     f\left(\frac{\xi}{2}\right)}d\xi, \qquad x\in\group\,.
\end{equation}
\end{definition}
\subsection{Sobolev and Bessel potential spaces on  $\group$}

For a function of polynomial growth at infinity $|a(x)|\leqslant C\langle
x\rangle^N:=C(1+|x|^2)^{N/2}$ for some constant $C>0$ and some integer $N\in\N_0$, by
$a_{\R^n}(D)$ we denote the Fourier convolution operator on the Euclidean
space $\R^n$, defined by
 \[
a_{\R^n}(D)\varphi(t):=\FTEucl^{-1}a\FTEucl\varphi(t), \qquad
     t\in\R^n,\quad \varphi\in\mathcal{S}(\R^n),
 \]
and $a(\xi)$ is called its symbol (cf. \cite{Du79}, where one used a
different notation $W^0_a=a_{\R^n}(D)$). Let $\Delta_{\mathbb{R}^n}$ be the standard Laplacian on $\mathbb{R}^n,$ given by
\begin{equation}
    \Delta_{\mathbb{R}^n}=-\sum_{j=1}^n\partial^2_{x_j}.
\end{equation}
The operator $ \Delta_{\mathbb{R}^n}$ admits a self-adjoint extension on $L^2(\R^n).$
\begin{definition}[Sobolev spaces on $\mathbb{R}^n$]
The notation $H^s_p(\R^n)$  with $s\in\R,\;1\leqslant p\leqslant\infty,$ refers to the Bessel potential space on the Lie group (Euclidean space) $\R^n$, which represents the closure of the Schwartz space $\mathcal{S}(\R^n)$ with respect of the norm
 \begin{eqnarray}\label{e6.1}
\|f\,\|_{{H}_p^s(\R^n)} \,:=\, \|(1-\Delta_{\mathbb{R}^n})^{\frac{s}{2}}f\|_{ {L}^p(\R^n)}<\infty.
 \end{eqnarray}
\end{definition}
\begin{remark}[Sobolev spaces on $L^2(\R^n)$]
We will apply the standard convention and for $p=2$ use the notation $H^s(\R^n)$ for the Hilbert space ${H}^s_2(\R^n)$, dropping the subscript index $p=2$. Due to the classical Parseval's equality for $\R^n$ the formula
 \begin{eqnarray}\label{e6.2}
\|f\|_{{H}^s(\R^n)}:=\left(\int\limits_{\R^n}\left|(1
     +|\xi|^2)^{s/2}\FTEucl\,f(\xi)\right|^2\,d\xi\right)^{1/2}
 \end{eqnarray}
provides an equivalent norm in the Hilbert space ${H}^s(\R^n)$.
\end{remark}
\begin{remark}
It is well known that a partial derivative $(\partial^\alpha\varphi)(t)=(a^\alpha(D)\varphi)(t)$ is a convolution operator and its symbol is $a(\xi):=(-i\xi)^\alpha$ for arbitrary multi-index $\alpha\in\N^n$:
\begin{eqnarray}\label{e6.3}
(\FTEucl\partial^\alpha\varphi)(\xi)=(-i\xi)^\alpha(\FTEucl
     \varphi)(\xi), \qquad \xi\in\R^n.
\end{eqnarray}
\end{remark}
\begin{remark}[Sobolev spaces of integer order]
For an integer $m\in\N$ and an arbitrary $1\leqslant p\leqslant\infty$ the
Bessel potential space coincides with the Sobolev space
${H}^m_p(\R^n)={W}_p^m (\R^n)$ and an equivalent  norm
is defined as follows (cf. \cite{Tr95})
 \begin{eqnarray*}
\| f\,\|_{{W}_p^m(\R^n)}:=\left(\sum\limits_{|\alpha|\leqslant
   m}\|\partial^\alpha f \|_{L^p(\R^n)}\right)^{1/p}
 \end{eqnarray*}
with the usual ${\rm ess}\,\sup$-norm modification for $p=\infty$:
 \[
\|f\,\|_{{W}_\infty^m(\R^n)}:=\,\sum\limits_{|\alpha|\leqslant
     m}{\rm ess}\,\sup\limits_{t\in\R^n}|\partial^\alpha f(t )|\, .
 \]
\end{remark}
\begin{definition}[The class of $L^p(\R^n)$-multipliers] For $1\leqslant p\leqslant\infty$ by $\mathfrak{M}_p(\R^n)$ we denote
the set of all symbols $a(\xi)$, $\xi\in\R^n$, for which the convolution
operator $a_{\R^n}(D)\;:\;\mathcal{S}(\R^n)\to\mathcal{S}^\prime(\R^n)$ extends to a bounded operator
 \[
a_{\R^n}(D)\;:\;L^p(\R^n) \rightarrow L^p(\R^n).
 \]
Symbols in the class  $\mathfrak{M}_p(\R^n)$ are called $L^p$-multipliers. 
\end{definition}
\begin{remark}
Note that $\mathfrak{M}_p(\R^n)$ is a Banach algebra with the induced operator norm
$\|a_{\R^n}(D)\|_{\mathscr{B}(L^p(\R^n)}$, because the composition of
operators from this class has the property  $$a_{\R^n}(D)b_{\R^n}(D)=(ab)_{\R^n}(D).$$ 
For a function of polynomial growth $a(\xi)$, $\xi\in\R^n,$ the convolution
operator $a_\group(\fD)$ on the group $G$ is defined in a standard way:
 \[
a_\group(\fD)\varphi(x):=\FTGrp^{-1}a\FTGrp\varphi(x), \qquad
     x\in\group=(-1,1)^n,\quad \varphi\in\mathcal{S}(\group),
 \]
and $a(\xi)$ is called its symbol (cf. \cite{Du22}, where is used a
different notation $W^0_{a,G}=a_G(\fD)$).
\end{remark}
\begin{remark}[$W^0_{a,G}=a_G(\fD)$] We adopt the notation $W^0_{a,G}=a_G(\fD)$ through this work. We emphasize the use of this notation in Section \ref{Spectral:2}, specially if the reader is familiar with the references \cite{Du76,Du79}.
\end{remark}
\begin{remark}[Partial derivatives on $G$]
Note that according to \eqref{eq:derivativeTransform1} and
\eqref{eq:derivativeTransform2}   the counterpart of the derivative
$\partial^\alpha=\partial_1^{\alpha_1}\cdots\partial_n^{\alpha_n}$ on
$\R^n$ is the derivative $\fD^\alpha_\group:=\fD_1^{\alpha_1}
\cdots\fD_n^{\alpha_n}$ on the Lie group $\group $ and both of them
are convolution operators with the same symbols $(-i\xi)^\alpha
=(-i\xi_1)^{\alpha_1}\cdots(-i\xi_n)^{\alpha_n}$.
\end{remark}
\begin{remark}[The Laplacian on $G$]
Note that in view of \eqref{eq:derivativeTransform2}, the Laplacian on $G$ is given by
\begin{equation}
    \Delta_{G}=-\sum_{j=1}^n  \pDiffGrpTo{x_{j}}^2.
\end{equation}
The operator $\Delta_{G}$ admits a self-adjoint extension on $L^2(G).$ 
\end{remark}
\begin{definition}[Sobolev spaces on $G$]
Arguing as in the Euclidean case, the notation ${H}^s_p(\group)$  with $s\in\R,\;1
\leqslant p\leqslant\infty$ refers to the Bessel potential space on
the Lie group $\group$, where the norm is defined as follows (cf. \eqref{e6.1}
 \begin{eqnarray}\label{e6.4}
\|f\,\|_{{H}_p^s(\group)} \,:=\, \|(1-\Delta_G)^{\frac{s}{2}}f
    \|_{L^p(\group)}<\infty.
 \end{eqnarray}
\end{definition}

\begin{remark}[Sobolev spaces on $L^2(G)$]
We follow the standard notation and we write ${H}^s(\group)$ for the space
${H}^s_2(\group)$ in the case $p=2$. Due to the Parseval's equality for the group $G$
 \begin{eqnarray}\label{e6.5}
\scal\FTGrp\varphi,\FTGrp\psi\scar_{L^2(\group)}=\int\limits_{\R^n}(\FTGrp
     \varphi)(\xi)\overline{(\FTGrp\psi)(\xi)}d\xi \nonumber=\pi^n\int\limits_G\varphi(y)\overline{\psi(y)}d\mu_\group y
     =:\pi^n\scal\varphi,\psi\scar_{L^2(\group)}
 \end{eqnarray}
the following
 \begin{eqnarray*}
\|f\,\|_{{H}^s(\group)}:=\left(\int\limits_{\mathbb{R}^n}(1
    +|\xi|^2)^{s}|[\FTGrp\,f](\xi)|^2\,d\xi\right)^{1/2}
 \end{eqnarray*}
defines an equivalent norm on ${H}^s(\group)$.
\end{remark}

\begin{remark}[Sobolev spaces of integer order]
For an integer $s=m\in \mathbb{N}$ and arbitrary $1\leqslant p\leqslant\infty$
the Bessel potential space coincides with the Sobolev space
${H}^m_p(\group)={W}^m_p(\group)$
and an equivalent  norm is defined as follows
 \[
\| f\|_{{W}_p^m(\group)}:=\,\left(
    \sum\limits_{|\alpha|\leqslant m}\|\fD^\alpha_\group f\|_{{L}^p(G)}^p\right)^{1/p},
 \]
for $1\leqslant p<\infty$, while for $p=\infty$ is used the usual ${\rm ess}\,\sup$-norm modification:
 \[
\|f\,\|_{{W}_\infty^m(\group)}:=\,\sum\limits_{|\alpha|
\leqslant m}{\rm ess}\,\sup\limits_{x\in\R^n}|\fD^\alpha_\group f(x)|\, .
 \]
\end{remark}
We remind that $\mathscr{B}(\fB_1,\fB_2)$ denotes the space of all linear bounded operators $T:\fB_1\rightarrow\fB_2$ between the Banach spaces $\fB_1$ and $\fB_2.$

The next Proposition \ref{p6.1} is proved in \cite{Du22} for $n=1$, but the proof for arbitrary $n$ is similar.
 %
\begin{proposition}\label{p6.1} For arbitrary $1\leqslant p\leqslant\infty$ and $s\in\R$ a convolution operator
 \[
a_{\R^n}(D)\;:\;{H}_p^s(\R^n) \rightarrow{H}_p^s(\R^n).
 \]
is bounded if and only if the operator
 \[
a_\group(\fD)\;:\;{H}_p^s(G) \rightarrow {H}_p^s(G)
 \]
is bounded and if and only if $a\in\mathfrak{M}_p(\R^n)$ (i.e. $a$ is an $ L^p $-multiplier).
 \end{proposition}
\begin{remark}
In other words, multiplier classes for the spaces ${H}_p^s(\R^n)$ and
${H}_p^s(G)$ are independent of the parameter $s$ and both coincide with $\mathfrak{M}_p(\R^n)$.
\end{remark}
Based on formulae \eqref{e6.3} and the isomorphism
\begin{eqnarray}\label{e6.6}
t_*&:&{H}_p^s(\R^n)\longrightarrow H^s_p(G),\qquad 1\leqslant p\leqslant\infty,\quad s\in\R,
 \end{eqnarray}
the following is proved (cf. \cite{Du79,Tr95} for the case
${H}_p^s(\R^n)={W}_p^m(\R^n)$).
 %
\begin{proposition}\label{p6.2} For arbitrary $1<p<\infty$ and an integer
$m\in\N_0$ the Bessel potential space $H^m_p(G)$
and the Sobolev space ${W}^m_p(G)$ have equivalent
norms and are topologically isomorphic.
 \end{proposition}

By applying the isomorpism \eqref{e6.5} we can also justify the following
propositions, proved in \cite[\S\,2.4.2]{Tr95} for the Bessel potential
space on the Euclidean space $\R^n$.
 %
\begin{proposition}\label{p6.3}
Let $s_0,s_1,r_0,r_1\in\R, \quad 1\leqslant p_0,p_1,q_0,q_1<\infty,$
$0<\ t<1$ and
\[
\begin{array}{c}
\dst\frac1p= \frac{1-\theta  }{p_0}+ \frac{\theta}{p_1},  \quad
     \dst\frac1q=\frac{1-\theta }{q_0}+ \frac{\theta}{q_1},\quad
     s=(1-\theta)s_0+\theta s_1\, ,\quad r=(1-\theta)r_0+\theta r_1\, .
\end{array}
\]

If $A\in\mathscr{B}_j:=\mathscr{B}(H^{s_j}_{p_j}(G),
H^{r_j}_{q_j}(G))$, $j=0,1$, then $A$ is bounded
between the interpolated spaces $A\in\mathscr{B}:=\mathscr{B}(H^s_p(G),H^r_q(G))$ and the norm is
estimated as follows
$$
 \|A\|_{\mathscr{B}}\leqslant\|A\|_{\mathscr{B}_0}^{1-\theta}\|A\|_{ \mathscr{B}_1}^\theta.
$$
 \end{proposition}
 %
\begin{proposition}[see \cite{Du79}]\label{p6.4} Let $s,r\in\R$,
$1\leqslant p\leqslant\infty$. The Bessel potential operator
 \begin{eqnarray}\label{e6.7}
(1-\Delta_G)^{\frac{r}{2}}
     \;:\;\bH_p^s(G)\rightarrow
     \bH_p^{s-r}(G).
 \end{eqnarray}
 \end{proposition}

\section{\bf Pseudo-differential operators on \texorpdfstring{$\group = (-1,1)^n$}{group}}\label{HClassesG}

\subsection{The global H\"ormander classes}

Let \(0 \leqslant \rho, \delta \leqslant 1\) and \(m \in \R\). Our goal is to give a useful definition of (global) symbol classes on \(G\). To this end we start from a symbol \( a=a(t,\xi) \in \symbClassOn{m}{\rho}{\delta}{\R^n \times \R^n}\).
\begin{remark}[H\"ormander classes on $\R^n$ v.s. H\"ormander classes on $G$]
In the same spirit of \cref{subsection:funcSpaces} we consider \(\symbToGrp{a}(x,\xi) := a(t(x),\xi) \in \CinftyOn{\group \times \R^n}\). This function satisfies for every \(\alpha, \beta \in \N_{0}^{n}\) and every \((x,\xi) \in \group \times \R^n\) the symbolic estimates
\begin{equation}\label{ineq:symbIneqGrp}
    \abs{\pDiffGrpToOrd{x}{\beta} \pDiffToOrd{\xi}{\alpha} \symbToGrp{a}(x,\xi)}
    = \abs{(\pDiffToOrd{t}{\beta} \pDiffToOrd{\xi}{\alpha}a)(t(x),\xi)}
    \leqslant C_{\alpha,\beta} (1+\abs{\xi})^{m-\rho\abs{\alpha}+\delta\abs{\beta}}.
\end{equation}
Conversely, if a smooth function \(a \in \CinftyOn{\group \times \R^{n}}\) satisfies the symbolic estimates \eqref{ineq:symbIneqGrp}, then we obtain for \(\symbToRn{a}(t,\xi) := a(x(t),\xi) \in \CinftyOn{\R^{n} \times \R^{n}}\) the inequalities
\[\abs{\pDiffToOrd{t}{\beta} \pDiffToOrd{\xi}{\alpha} \symbToRn{a}(t,\xi)} = \abs{(\pDiffGrpToOrd{x}{\beta} \pDiffToOrd{\xi}{\alpha} a)(x(t),\xi)} \leq C_{\alpha,\beta} (1+\abs{\xi})^{m-\rho\abs{\alpha}+\delta\abs{\beta}}.\]
Hence, we define the symbol class \(\symbClassOn{m}{\rho}{\delta}{\group \times \R^{n}}\) as the space of smooth functions \(a \in \CinftyOn{\group \times \R^{n}}\) that satisfy for all \((x,\xi) \in \group \times \R^{n}\) and all \(\alpha, \beta \in \N_{0}^{n}\) the symbolic estimates \[\abs{\pDiffToOrd{\xi}{\alpha} \pDiffGrpToOrd{x}{\beta} \sigma(x,\xi)} \leqslant C_{\alpha,\beta}(1+\abs{\xi})^{m-\rho\abs{\alpha}+\delta\abs{\beta}},\]
where \(C_{\alpha,\beta} > 0\) are some positive constants.

Moreover, it can easily be checked that \(\symbToRn{\left(\symbToGrp{a}\right)} = a\) for \(a \in \symbClassOn{m}{\rho}{\delta}{\R^n \times \R^n}\), and similarly \(\symbToGrp{\left(\symbToRn{a}\right)} = a\) for \(a \in \symbClassOn{m}{\rho}{\delta}{\group \times \R^n}\). Thus, there is a canonical bijection between \(\symbClassOn{m}{\rho}{\delta}{\R^n \times \R^n}\) and \(\symbClassOn{m}{\rho}{\delta}{\group \times \R^n}\).
\end{remark}

\subsection{The quantisation formula}

Let \(\sigma \in \symbClassOn{m}{\rho}{\delta}{\R^{n} \times \R^{n}}\) for some \(m \in \R\) and \(0 \leq \rho, \delta \leq 1\). The symbol \(\sigma\) determines a pseudo-differential operator expressed as
\[\opRnOf{\sigma}f(t) = \int\limits_{\R^{n}}{e^{2 \pi i t \cdot \xi} \, \sigma(t,\xi) \, \FTEucl{f}(\xi)}{d\xi}, \qquad f \in \SchwartzSpaceOn{\R^{n}}.\]
We look for an expression of pseudo-differential operators on \(\group\) now. So, let \(f \in \SchwartzSpaceOn{\group}\). Then \(\pull{x}f \in \SchwartzSpaceOn{\R^n}\) so that \(\pull{t} \opRnOf{\sigma} \pull{x} f \in \SchwartzSpaceOn{\group}\). We compute an expression for this function for \(y \in \group\) with the help of \eqref{eq:FTToGrp}:
\begin{align}
    \pull{t} \left[\opRnOf{\sigma} \left(\pull{x} f\right)\right](y) &= \int\limits_{\R^{n}}{e^{2 \pi i t(y) \cdot \xi} \, \sigma(t(y),\xi) \, \FTEucl[\pull{x}f](\xi)}{d\xi} \nonumber \\
    &= \int\limits_{\R^{n}}{\left(\frac{1-y}{1+y}\right)^{i \pi \xi} \symbToGrp{\sigma}(y,\xi) \, \FTGrp{f}(\xi/2)}{d\xi}.\label{eq:OpToGrp}
\end{align}
Hence, we define the pseudo-differential operator \(\sigma(x,\mathfrak{D})\) associated with the symbol \(\sigma \in \symbClassOn{m}{\rho}{\delta}{\group \times \R^n}\) by
\begin{equation}\label{e2.3}
    \sigma(x,\mathfrak{D})f(x) = \int\limits_{\R^n}{\left(\frac{1-x}{1+x}\right)^{i \pi \xi} \sigma(x,\xi) \, \FTGrp f\left(\frac\xi2\right)}{d\xi}, \qquad f \in \SchwartzSpaceOn{\group}.
\end{equation} Note that the argument $\xi/2$ in \eqref{e2.3} comes from the Fourier inversion formula \eqref{FIF}. On the other hand, we denote by
\begin{equation}
    \PsiDOHor{m}{\rho}{\delta}{\group \times \R^n} = \{\sigma(x,\mathfrak{D}) : \sigma \in \symbClassOn{0}{\rho}{\delta}{\group \times \R^n}\}
\end{equation}
the family of (global) pseudo-differential operators on $\group$ with order $m$ and of $(\rho,\delta)$-type.
Moreover, it follows from \eqref{eq:OpToGrp} that there is a bijective correspondence between pseudo-differential operators on \(\R^n\) and on \(\group\) given by
\begin{equation}\label{eq:PsiDOCorrespondence}
    \pull{t} \opRnOf{\sigma} \pull{x} = \sigma^G(x,\mathfrak{D}).
\end{equation}
Note that this is not really surprising as this is the change of variables formula for pseudo-differential operators. 

We summarise the  interplay between the H\"ormander classes on $\R^n$ with the global ones on $G$ in the following theorem.
\begin{corollary}
    Let $m\in \R$ and $0\leqslant \delta,\rho\leqslant 1$. Then, $A : \SchwartzSpaceOn{G} \to \SchwartzSpaceOn{G}$ is a pseudo-differential operator in the class $\PsiDOHor{m}{\rho}{\delta}{\group \times \R^n}$ if and only if there exists a pseudo-differential operator
      $\tilde{A} : \SchwartzSpaceOn{\R^n} \to \SchwartzSpaceOn{\R^n}$ in the class $\PsiDOHor{m}{\rho}{\delta}{\R^n \times \R^n}$ such that the following diagram
    \begin{eqnarray}
        \begin{tikzpicture}[every node/.style={midway}]
      \matrix[column sep={10em,between origins}, row sep={4em}] at (0,0) {
        \node(R) {$\SchwartzSpaceOn{\R^n}$}  ; & \node(S) {$\SchwartzSpaceOn{\R^n}$}; \\
        \node(R/I) {$\SchwartzSpaceOn{\group}$}; & \node (T) {$\SchwartzSpaceOn{\group}$};\\
      };
      \draw[<-] (R/I) -- (R) node[anchor=east]  {$\pull{t}$};
      \draw[->] (R) -- (S) node[anchor=south] {$\tilde{A}$};
      \draw[->] (S) -- (T) node[anchor=west] {$\pull{t}$};
      \draw[->] (R/I) -- (T) node[anchor=north] {$A$};
    \end{tikzpicture}
    \end{eqnarray}commutes. It means, for any $f=f(t) \in \SchwartzSpaceOn{\R^n}$, $\pull{t} \circ \tilde{A} f = A\pull{t}f$. Moreover, if $\tilde{a}$ is the symbol of $\tilde{A}$, then the symbol of $a$ is given by $$a(x,\xi) \equiv \symbToGrp{\tilde{a}}(x,\xi)=\tilde{a}(t(x),\xi),$$ for all $(x,\xi)\in \group \times \R^n$.
\end{corollary}

\begin{proof}
    The bijective correspondence \eqref{eq:PsiDOCorrespondence} between \(\PsiDOHor{m}{\rho}{\delta}{\R^n \times \R^n}\) and \(\PsiDOHor{m}{\rho}{\delta}{\group \times \R^n}\) yields \(\pull{t} \tilde{A} \pull{x} = A\), which is equivalent to \(\pull{t} \tilde{A} = A \pull{t}\).
    This means that \(\pull{t} \tilde{A} f = A \pull{t} f\) for all \(f \in \SchwartzSpaceOn{\R^n}\). This bijective correspondence is indeed such that the symbols are related by
    \(a(x,\xi) = \symbToGrp{\tilde{a}}(x,\xi)\).
\end{proof}

\subsection{Composition of pseudo-differential operators}
Here we study the composition of pseudo-differential operators. We show that the classes $ \symbClassOn{m}{\rho}{\delta}{\group \times \R^n}$ are stable  under the composition of operators.
\begin{theorem}
    Let \(0 \leqslant \delta < \rho \leqslant 1\). Let \(a \in \symbClassOn{m_1}{\rho}{\delta}{\group \times \R^n}\) and \(b \in \symbClassOn{m_2}{\rho}{\delta}{\group \times \R^n}\). Then there exists a symbol \(c \in \symbClassOn{m_1 + m_2}{\rho}{\delta}{\group \times \R^n}\) such that \(c(x,\mathfrak{D}) = a(x,\mathfrak{D}) \circ b(x,\mathfrak{D})\). Moreover, we have the asymptotic formula
    \[\forall (x,\xi)\in \group \times \R^n:  c(x,\xi) \sim \sum_{\alpha} \frac{(2 \pi i)^{-\abs{\alpha}}}{\alpha!} (\pDiffToOrd{\xi}{\alpha} a(x,\xi)) (\pDiffGrpToOrd{x}{\alpha} b(x,\xi)),\]
    in the sense that for any $N \in \N,$
\[(x,\xi) \mapsto c(x,\xi) - \sum_{\abs{\alpha}<N} \frac{(2 \pi i)^{-\abs{\alpha}}}{\alpha!} (\pDiffToOrd{\xi}{\alpha} a(x,\xi)) (\pDiffGrpToOrd{x}{\alpha} b(x,\xi)) \in \symbClassOn{m-(\rho-\delta)N}{\rho}{\delta}{\group \times \R^n}.\]
\end{theorem}

\begin{proof}
    Suppose \(a \in \symbClassOn{m_1}{\rho}{\delta}{\group \times \R^n}\) and \(b \in \symbClassOn{m_2}{\rho}{\delta}{\group \times \R^n}\). Then \(\symbToRn{a} \in \symbClassOn{m_1}{\rho}{\delta}{\R^n \times \R^n}\) and \(\symbToRn{b} \in \symbClassOn{m_2}{\rho}{\delta}{\R^n \times \R^n}\). By the composition formula in \(\R^n\) 
    there exists a symbol \(c \in \symbClassOn{m_{1}+m_{2}}{\rho}{\delta}{\R^n \times \R^n}\) such that \(\opRnOf{c} = \opRnOf{\symbToRn{a}} \circ \opRnOf{\symbToRn{b}}\), and we have the asymptotic expansion
    \begin{equation}\label{eq:asympC}
        c(t,\xi) \sim \sum_{\alpha} \frac{(2 \pi i)^{-\abs{\alpha}}}{\alpha!} (\pDiffToOrd{\xi}{\alpha} \symbToRn{a}(t,\xi)) (\pDiffToOrd{t}{\alpha} \symbToRn{b}(t,\xi)).
    \end{equation}
    With the help of \eqref{eq:PsiDOCorrespondence} we obtain
    \[c(x,\mathfrak{D})= \pull{t} \opRnOf{c} \pull{x} = \pull{t} \opRnOf{\symbToRn{a}} \pull{x} \pull{t} \opRnOf{\symbToRn{b}} \pull{x} = a(x,\mathfrak{D}) \circ b(x,\mathfrak{D}),\]
    where we have applied that \(\pull{x}\) and \(\pull{t}\) are each other's inverse.

    Let \(N \in \N\). Using the formula \(\symbToGrp{\sigma}(x,\xi) = \sigma(t(x),\xi)\) for any \(\sigma \in \symbClassOn{m}{\rho}{\delta}{\R^n \times \R^n}\), we deduce from \eqref{eq:asympC} that
    \[\symbToGrp{c}(x,\xi) - \sum_{\abs{\alpha}<N} \frac{(2 \pi i)^{-\abs{\alpha}}}{\alpha!} (\pDiffToOrd{\xi}{\alpha} a(x,\xi)) (\pDiffGrpToOrd{x}{\alpha} b(x,\xi)) \in \symbClassOn{m-(\rho-\delta)N}{\rho}{\delta}{\group \times \R^n}.\]
    This completes the proof.
\end{proof}

\subsection{The adjoint of a pseudo-differential operator} Here we study the adjoints of the pseudo-differential operators in the classes $ \symbClassOn{m}{\rho}{\delta}{\group \times \R^n}$. We show that theses classes  are stable  under adjoints.
\begin{theorem}
    Let \(0 \leqslant \delta < \rho \leqslant 1\). Let \(a \in \symbClassOn{m}{\rho}{\delta}{\group \times \R^n}\). Then there exists a symbol \(\adj{a} \in \symbClassOn{m}{\rho}{\delta}{\group \times \R^n}\) such that \(a^\ast(x,\mathfrak{D}) = a(x,\mathfrak{D})^\ast\). Moreover, we have the asymptotic formula
    \[\forall (x,\xi)\in \group \times \R^n:  \adj{a}(x,\xi) \sim \sum_{\alpha} \frac{(2 \pi i)^{-\abs{\alpha}}}{\alpha!} \pDiffToOrd{\xi}{\alpha} \pDiffGrpToOrd{x}{\alpha} \conj{a(x,\xi)},\]
    in the sense that for any $N \in \N$,
    \[(x,\xi) \mapsto \adj{a}(x,\xi) - \sum_{\abs{\alpha}<N} \frac{(2 \pi i)^{-\abs{\alpha}}}{\alpha!} \pDiffToOrd{\xi}{\alpha} \pDiffGrpToOrd{x}{\alpha} \conj{a(x,\xi)} \in \symbClassOn{m-(\rho-\delta)N}{\rho}{\delta}{\group \times \R^n}.\]
\end{theorem}

\begin{proof} Let \(a \in \symbClassOn{m}{\rho}{\delta}{\group \times \R^n}\). Then \(\symbToRn{a} \in \symbClassOn{m}{\rho}{\delta}{\R^n \times \R^n}\) so that the adjoint formula in \(\R^n\) 
\begin{equation}\label{eq:asympAdjA}
        \adj{\left(\symbToRn{a}\right)}(t,\xi) \sim \sum_{\alpha} \frac{(2 \pi i)^{-\abs{\alpha}}}{\alpha!} \pDiffToOrd{\xi}{\alpha} \pDiffToOrd{t}{\alpha} \conj{\symbToRn{a}(t,\xi)}.
\end{equation}
{Due to \eqref{eq:PsiDOCorrespondence} it follows that}
\[
\opRnOf{\symbToRn{a}} = \pull{x} a(x,\mathfrak{D}) \pull{t}.
\]
Let \(\adj{a} := \symbToGrp{\left[\adj{\left(\symbToRn{a}\right)}\right]}\). Using \eqref{eq:innerProducts} we find on the one hand for every \(f,g \in \SchwartzSpaceOn{\R^n}\) that
\begin{align*}
        \inProdIn{\adj{\opRnOf{\symbToRn{a}}}f}{g}{L^{2}(\R^n)}
        &= \inProdIn{f}{\opRnOf{\symbToRn{a}}g}{L^{2}(\R^n)}
        = \inProdIn{f}{\pull{x}a(x,\mathfrak{D})\pull{t}g}{L^{2}(\R^n)} \\
        &= \inProdIn{\pull{t}f}{a(x,\mathfrak{D})\pull{t} g}{L^{2}(\group)}
        = \inProdIn{a(x,\mathfrak{D})^\ast\pull{t} f}{\pull{t} g}{L^{2}(\group)},
\end{align*}
while on the other
    \begin{align*}
        \inProdIn{\adj{\opRnOf{\symbToRn{a}}}f}{g}{L^{2}(\R^n)}
        = \inProdIn{\opRnOf{\adj{\left(\symbToRn{a}\right)}} f}{g}{L^{2}(\R^n)}
        &= \inProdIn{\pull{t} \opRnOf{\adj{\left(\symbToRn{a}\right)}} \pull{x} \pull{t} f}{\pull{t} g}{L^{2}(\group)} \\
        &= \inProdIn{a^\ast(x,\mathfrak{D}\pull{t} f}{\pull{t} g}{L^{2}(\group)}.
    \end{align*}
    Since \(\pull{t}\) is a bijection from \(\SchwartzSpaceOn{\R^n}\) to \(\SchwartzSpaceOn{\group}\), it follows that \(a^\ast(x,\mathfrak{D})=a(x,\mathfrak{D})^\ast\).

    Applying the bijective correspondence between \(\symbClassOn{m}{\rho}{\delta}{\R^n \times \R^n}\) and \(\symbClassOn{m}{\rho}{\delta}{\group \times \R^n}\) to \eqref{eq:asympAdjA}, we get
    \[\adj{a}(x,\xi) - \sum_{\abs{\alpha}<N} \frac{(2 \pi i)^{-\abs{\alpha}}}{\alpha!} \pDiffToOrd{\xi}{\alpha} \pDiffGrpToOrd{x}{\alpha} \conj{a(x,\xi)} \in \symbClassOn{m-(\rho-\delta)N}{\rho}{\delta}{\group \times \R^n}\]
    for any \(N \in \N\), which proves the theorem.
\end{proof}
\subsection{Construction of parametrices} Now we will prove that the classes $S^{m}_{\rho,\delta}(G\times \R^n)$ admit the construction of parametrices. For this, the condition of ellipticity for the symbol is required. A symbol $a\in {S}^{m}_{\rho,\delta}(G\times \R^n)$ is elliptic of order $m,$ if there exists $R>0,$ such that the inequality condition in \eqref{Iesparametrix} holds.  For the corresponding operator $A=a(x,D)$ we compute the parametrix in the next theorem.

\begin{theorem}
 \label{IesTParametrix} Let $m\in \R,$ and let $0\leqslant \delta<\rho\leqslant 1.$  Let  $a=a(x,\xi)\in {S}^{m}_{\rho,\delta}(G\times \R^n).$  Assume also that $a(x,\xi)\neq 0$  for $|\xi|\geq R,$ for some $R>0,$ and it satisfies
\begin{equation}\label{Iesparametrix}
   \inf_{(x,\xi)\in G\times \{\xi\in \R^n:|\xi|\geq R\}} \vert a(x,\xi)\vert\geq C(1+|\xi|)^{m}.
\end{equation}Then, there exists $B\in \Psi^{-m}_{\rho,\delta}(G\times \R^n),$ such that $$ AB-I,BA-I\in \Psi^{-\infty}(G\times \R^n):=\bigcap_{s\in \R} \Psi^{m}(G\times \R^n).$$ Moreover, the symbol $\tau(x,\xi)$ of $B$ satisfies the following asymptotic expansion
\begin{equation}\label{AE}
    \tau(x,\xi)\sim \sum_{N=0}^\infty\tau_{N}(x,\xi),\,\,\,(x,\xi)\in G\times \R^n,
\end{equation}where $\tau_{N}\in {S}^{-m-(\rho-\delta)N}_{\rho,\delta}(G\times \R^n)$ obeys to the recursive  formula
\begin{equation}\label{conditionelip}
    \tau_{N}(x,\xi)=-a(x,\xi)^{-1}\left(\sum_{k=0}^{N-1}\sum_{|\gamma|=N-k}(\partial_{\xi}^\gamma a(x,\xi))(\mathfrak{D}_{x}^{\gamma}\tau_{k}(x,\xi))\right),\,\,N\geqslant 1,
\end{equation}with $ \tau_{0}(x,\xi)=a(x,\xi)^{-1}.$
\end{theorem}
\begin{proof}
The idea is to find a symbol $\tau$ such that if $\mathcal{I}=AB,$ then $\mathcal{I}-I$ is a smoothing operator, where
\begin{align*}
  \widehat{\mathcal{I}}(x,\xi)  \sim \sum_{|\alpha|= 0}^\infty( \partial_{\xi}^{\alpha} a(x,\xi))(\partial_{X}^{(\alpha)} \tau(x,\xi)).
\end{align*} Here $\widehat{\mathcal{I}}(x,\xi)$ denotes the symbol of $\mathcal{I}.$ The asymptotic expansion means that for every $N\in \mathbb{N},$
\begin{align*}
    &\partial_{\xi}^{\alpha_\ell}\mathfrak{D}_{x}^{\alpha_\ell}\left(\widehat{\mathcal{I}}(x,\xi)-\sum_{|\alpha|\leqslant N}  ( \partial_{\xi}^{\alpha} a(x,\xi))(\mathfrak{D}_{x}^{\alpha} \tau(x,\xi))  \right)\\
    &\hspace{3cm}\in {S}^{-(\rho-\delta)(N+1)-\rho\ell+\delta|\beta|}_{\rho,\delta}(G\times \R^n),
\end{align*}for every $\alpha_\ell\in \mathbb{N}_0$ of order $\ell\in\mathbb{N}_0,$ where $\tau$ is requested to satisfy the asymptotic expansion  \eqref{AE}. So, formally we can write
\begin{align*}
  \widehat{\mathcal{I}}(x,\xi) &\sim \sum_{|\alpha|= 0}^\infty( \partial_{\xi}^{\alpha} a(x,\xi))(\mathfrak{D}_{x}^{\alpha} \tau(x,\xi))=\sum_{|\alpha|= 0}^\infty\sum_{N=0}^{\infty}( \partial_{\xi}^{\alpha} a(x,\xi))(\mathfrak{D}_{x}^{\alpha} \tau_N(x,\xi)).
\end{align*}
  Observe {the fact that} $\tau_{0}\in {S}^{-m}_{\rho,\delta}(G\times \R^n) $    follows from the hypothesis. Now, one can check easily that $\tau_{N}\in {S}^{-m-(\rho-\delta)N}_{\rho,\delta}(G\times \R^n),$ for all $N\geqslant 1$ by using induction. Consequently,
  \begin{align*}
      \tau(x,\xi)- \sum_{j=0}^{N-1}\tau_{j}(x,\xi)\in {S}^{-m-(\rho-\delta)N}_{\rho,\delta}(G\times \R^n).
  \end{align*}This analysis allows us to deduce that
  \begin{align*}
    \widehat{\mathcal{I}}(x,\xi)-\sum_{k=0}^{N-1}\sum_{|\gamma|<N}(\partial_{\xi}^\gamma a(x,\xi))(\mathfrak{D}_{x}^{\gamma}\tau_{k}(x,\xi))\in {S}^{-(\rho-\delta)N}_{\rho,\delta}(G\times \R^n).
  \end{align*}On the other hand, observe that
  \begin{align*}
   &\sum_{k=0}^{N-1}\sum_{|\gamma|< N}(\partial_{\xi}^{\gamma} a(x,\xi))(\mathfrak{D}_{x}^{\gamma}\tau_{k}(x,\xi))\\
  & =1+\sum_{k=1}^{N-1}\left(a(x,\xi)\tau_{k}(x,\xi)+\sum_{ |\gamma|\leqslant N,\,|\gamma|\geqslant 1}(\partial_{\xi}^{\gamma} a(x,\xi))(\mathfrak{D}_{x}^{\gamma}\tau_{k}(x,\xi))\right)\\
  & =1+\sum_{k=1}^{N-1}\left(a(x,\xi)\tau_{k}(x,\xi)+\sum_{ |\gamma|= N-j,\,j<k}(\partial_{\xi}^{\gamma} a(x,\xi))(\mathfrak{D}_{x}^{\gamma}\tau_{j}(x,\xi))\right)\\
  &\hspace{2cm}+\sum_{ |\gamma|+j\geqslant  N,\,|\gamma|<N,\,j<N}(\partial_{\xi}^{\gamma} a(x,\xi))(\mathfrak{D}_{x}^{\gamma}\tau_{k}(x,\xi))\\
   &=1+\sum_{ |\gamma|+j\geqslant  N,\,|\gamma|<N,\,j<N}(\partial_{\xi}^{\gamma} a(x,\xi))(\mathfrak{D}_{x}^{\gamma}\tau_{j}(x,\xi)),
  \end{align*}where we have used that
  $$ a(x,\xi)\tau_{k}(x,\xi)+\sum_{ |\gamma|= k-j,\,j<k}(\partial_{\xi}^{\gamma} a(x,\xi))(\mathfrak{D}_{x}^{\gamma}\tau_{j}(x,\xi))\equiv 0,  $$
  in view of \eqref{conditionelip}. Because, for $|\gamma|+j\geqslant  N,$ $(\partial_{\xi}^{\gamma} a(x,\xi))(\mathfrak{D}_{x}^{\gamma}\tau_{k}(x,\xi))\in {S}^{-(\rho-\delta)N}_{\rho,\delta}(G\times \R^n), $ it follows that
  \begin{align*}
      \sum_{k=0}^{N-1}\sum_{|\gamma|< N}(\partial_{\xi}^{\gamma} a(x,\xi))(\mathfrak{D}_{x}^{\gamma}\tau_{k}(x,\xi))-1\in  {S}^{-(\rho-\delta)N}_{\rho,\delta}(G\times \R^n),
  \end{align*}and consequently, $\widehat{\mathcal{I}}(x,\xi)-1\in  {S}^{-(\rho-\delta)N}_{\rho,\delta}(G\times \R^n),$ for every $N\in \mathbb{N}.$
 So, we have proved that  $AB-I\in {S}^{-\infty}(G\times \R^n). $ A similar analysis can be used to prove that $BA-I\in {S}^{-\infty}(G\times \R^n).$
\end{proof}

\section{\bf Boundedness properties for the H\"ormander classes $\Psi^{m}_{\rho,\delta}(G\times \R^n) $ }\label{mapping:properties}
In this section we discuss the mapping properties of pseudo-differential operators even in the border line case $0\leq\rho\leq \delta\leq 1,$ $\delta\neq 1.$
\subsection{The \texorpdfstring{$L^p$}{Lp}-theory: $L^p$-Fefferman theorem and $L^2$-Calder\'on-Vaillancourt theorem}

The following theorem provides sharp conditions for the $L^p$-boundedness of the pseudo-differential operators in the H\"ormander classes $\Psi^{m}_{\rho,\delta}(G\times \R^n).$

\begin{theorem}\label{MainTheorem}
 Let $A : \SchwartzSpaceOn{\group} \to \SchwartzSpaceOn{\group}$ be a pseudo-differential operator with symbol $\sigma \in \symbClassOn{-m}{\rho}{\delta}{\group \times \R^n}$, $0 \leqslant \delta \leqslant \rho \leqslant 1$, $\delta \neq 1$. Then, if $$m \geqslant m_{p}:= n(1-\rho)\abs{\frac{1}{p}-\frac{1}{2}},$$ where $1 < p < \infty$, then $A$ extends to a bounded operator from $L^p(\group)$ into $ L^p(\group)$.
\end{theorem}

\begin{proof}
Let $\tilde{A} : \SchwartzSpaceOn{\R^n} \to \SchwartzSpaceOn{\R^n}$ be the continuous linear operator such that the following diagram
\begin{eqnarray}
    \begin{tikzpicture}[every node/.style={midway}]
  \matrix[column sep={10em,between origins}, row sep={4em}] at (0,0) {
    \node(R) {$\SchwartzSpaceOn{\R^n}$}  ; & \node(S) {$\SchwartzSpaceOn{\R^n}$}; \\
    \node(R/I) {$\SchwartzSpaceOn{\group}$}; & \node (T) {$\SchwartzSpaceOn{\group}$};\\
  };
  \draw[<-] (R/I) -- (R) node[anchor=east]  {$\pull{t}$};
  \draw[->] (R) -- (S) node[anchor=south] {$\tilde{A}$};
  \draw[->] (S) -- (T) node[anchor=west] {$\pull{t}$};
  \draw[->] (R/I) -- (T) node[anchor=north] {$A$};
\end{tikzpicture}
\end{eqnarray}commutes. Then, $\tilde{A}$ is a pseudo-differential operator with symbol in the H\"ormander class $\symbClassOn{-m}{\rho}{\delta}{\R^n \times \R^n}$ and in view of Fefferman $L^p$-theorem (see \cite{Fefferman1973}),  the order condition $m \geqslant m_{p}:= n(1-\rho)\abs{\frac{1}{p}-\frac{1}{2}}$, implies that, for $1 < p < \infty$,
$$\tilde{A} : \SchwartzSpaceOn{\R^n} \subset L^p(\R^n) \to L^p(\R^n)$$ extends to a bounded operator. In view of the topological density of $\SchwartzSpaceOn{\R^n}$ into $L^p(\R^n),$ such an extension is unique and we can also denote it by $\tilde{A} : L^p(\R^n) \to L^p(\R^n)$. Now, there exists a unique operator $\dot{A} : L^p(\group) \to L^p(\group)$ such that the following diagram
\begin{eqnarray}
    \begin{tikzpicture}[every node/.style={midway}]
  \matrix[column sep={10em,between origins}, row sep={4em}] at (0,0) {
    \node(R) {$L^p(\R^n)$}  ; & \node(S) {$L^p(\R^n)$}; \\
    \node(R/I) {$L^p(\group)$}; & \node (T) {$L^p(\group)$};\\
  };
  \draw[<-] (R/I) -- (R) node[anchor=east]  {$\pull{t}$};
  \draw[->] (R) -- (S) node[anchor=south] {$\tilde{A}$};
  \draw[->] (S) -- (T) node[anchor=west] {$\pull{t}$};
  \draw[->] (R/I) -- (T) node[anchor=north] {$\dot{A}$};
\end{tikzpicture}
\end{eqnarray}commutes. It is clear that $\dot{A}|_{\SchwartzSpaceOn{\group}} = A.$ So, we have proven that $A$ admits a bounded extension on $L^p(\group)$.
\end{proof}
 With $p=2,$ $m=0$ and $0\leq \delta\leq \rho\leq 1,$ {$\delta\neq 1,$} in Theorem \ref{MainTheorem} we obtain the following analogue  of the Calder\'on-Vaillancourt theorem.
 \begin{corollary}\label{Calderon-Vaillancourth}
 Let $A : \SchwartzSpaceOn{\group} \to \SchwartzSpaceOn{\group}$ be a pseudo-differential operator with symbol $\sigma \in \symbClassOn{0}{\rho}{\delta}{\group \times \R^n}$, $0 \leqslant \delta \leqslant \rho \leqslant 1$, $\delta \neq 1$. Then, $A$ extends to a bounded operator from $L^2(\group)$ into $ L^2(\group)$.
\end{corollary}

\subsection{The sharp G\r{a}rding inequality and the Fefferman-Phong inequality} In this section we discuss the lower bounds for operators.

\begin{theorem}\label{S:Garding:Inequality}
 Let $A : \SchwartzSpaceOn{\group} \to \SchwartzSpaceOn{\group}$ be a pseudo-differential operator with symbol $a\in \symbClassOn{m}{\rho}{\delta}{\group \times \R^n}$ such that {$a(x,\xi)\geq 0$ for all $(x,\xi)\in G\times\R^n.$}
 Then the following sharp G\r{a}rding inequality
 \begin{equation}
     \forall u\in \mathcal{S}(G),\,\,\textnormal{Re}(a(x,D)u,u)_{L^2}\geq -C\Vert u\Vert_{H^{\frac{m-(\rho-\delta)}{2}}(G)}
 \end{equation}holds. Moreover, if $\sigma \in \symbClassOn{2}{1}{0}{\group \times \R^n}$ and $\sigma(x,\xi)\geq 0$ for all $(x,\xi)\in G\times\R^n ,$ then the Fefferman-Phong inequality
 \begin{equation}
     \forall u\in \mathcal{S}(G),\,\,\textnormal{Re}(a(x,D)u,u)_{L^2}\geq -C\Vert u\Vert_{L^2(G)}
 \end{equation} remains valid.
\end{theorem}

\begin{proof}
Let $\tilde{A} : \SchwartzSpaceOn{\R^n} \to \SchwartzSpaceOn{\R^n}$ be the continuous linear operator such that the following diagram
\begin{eqnarray}
    \begin{tikzpicture}[every node/.style={midway}]
  \matrix[column sep={10em,between origins}, row sep={4em}] at (0,0) {
    \node(R) {$\SchwartzSpaceOn{\R^n}$}  ; & \node(S) {$\SchwartzSpaceOn{\R^n}$}; \\
    \node(R/I) {$\SchwartzSpaceOn{\group}$}; & \node (T) {$\SchwartzSpaceOn{\group}$};\\
  };
  \draw[<-] (R/I) -- (R) node[anchor=east]  {$\pull{t}$};
  \draw[->] (R) -- (S) node[anchor=south] {$\tilde{A}$};
  \draw[->] (S) -- (T) node[anchor=west] {$\pull{t}$};
  \draw[->] (R/I) -- (T) node[anchor=north] {$A$};
\end{tikzpicture}
\end{eqnarray}commutes. Then, $\tilde{A}$ is a pseudo-differential operator with {the} symbol in the H\"ormander class $\symbClassOn{-m}{\rho}{\delta}{\R^n \times \R^n}$ and in view of the sharp G\r{a}rding inequality  for the classes $S^m_{\rho,\delta}(\R^n\times \R^n )$ (see H\"ormander \cite{Hormander1966}) {it has} the lower bound
\begin{equation}
     \forall \tilde{u}\in \mathcal{S}(\R^n),\,\,\textnormal{Re}(\tilde{A}\tilde{u},\tilde{u})_{L^2(\R^n)}\geq -C\Vert \tilde{u}\Vert_{H^{\frac{m-(\rho-\delta)}{2}}(\R^n)}.
 \end{equation}
Using that any $u\in \mathcal{S}(G)$ can be written in a unique way as $\tilde{u}=(t_{*})^{-1}u,$ and that $$t_{*}:L^2(\R^n)\rightarrow L^2(G),\,\, (t_{*})^{-1}:L^2(G)\rightarrow L^2(\R^n)$$ are isometries of Hilbert spaces, we have for $u\in \mathcal{S}(G)$ the identity of inner products
\begin{align*}
   \textnormal{Re}(\tilde{A}(t_{*})^{-1}u,(t_{*})^{-1}u)_{L^2(\R^n)} &= \textnormal{Re}( ( t_{*})^{-1} t_{*}\tilde{A}(t_{*})^{-1}u,(t_{*})^{-1}u)_{L^2(\R^n)}\\
   &=\textnormal{Re}( ( t_{*})^{-1} Au,(t_{*})^{-1}u)_{L^2(\R^n)}\\
   &=\textnormal{Re}(  Au,u)_{L^2(G)}.
\end{align*}
Since
\begin{equation*}
    \Vert u\Vert_{H^{\frac{m-(\rho-\delta)}{2}}(G)}=\Vert \tilde{u}\Vert_{H^{\frac{m-(\rho-\delta)}{2}}(\R^n)},
\end{equation*}we conclude that
\begin{align*}
    \textnormal{Re}(  Au,u)_{L^2(G)}\geq -C\Vert u\Vert_{H^{\frac{m-(\rho-\delta)}{2}}(G)}
\end{align*}as desired. On the other hand, if $m=2,$ and $(\rho,\delta)=(1,0),$ the Fefferman-Phong inequality (see \cite{FeffermanPhong78}) gives the lower bound
\begin{equation}
     \forall \tilde{u}\in \mathcal{S}(\R^n),\,\,\textnormal{Re}(\tilde{A}\tilde{u},\tilde{u})_{L^2(\R^n)}\geq -C\Vert \tilde{u}\Vert_{L^2(\R^n)}.
 \end{equation}By following the analysis in the first part of the proof, and the fact that $\Vert u\Vert_{L^2(G)}=\Vert \tilde{u}\Vert_{L^2(\R^n)}$ implies the following  Fefferman-Phong inequality
\begin{equation}
     \forall {u}\in \mathcal{S}(G),\,\,\textnormal{Re}(Au,u)_{L^2(G)}\geq -C\Vert u\Vert_{L^2(G)}.
 \end{equation}
The proof is complete.
\end{proof}

\section{\bf Sharp spectral properties for the classes \texorpdfstring{$\symbClassOn{m}{\rho}{\delta}{\group \times \R^n}$}{SymbolClass}}\label{Spectral:1}

In this section we discuss the spectral properties of operators. We investigate the compactness and the  Fredholmness of the operators.
\subsection{The Gohberg lemma and  compactness on $L^2(G)$} In the following theorem we characterise the $L^2(G)$-compactness of a sub-class of operators in $\Psi^{0}_{1,0}(G\times \R^n).$ To do so, we compute the distance of a pseudo-differential operator $A$ to the set {$\mathscr{C}(L^2(G)$) of all} compact operator on $L^2(G)$ and we prove that this distance inequality is sharp.

\begin{theorem}\label{Gohberg:Lemma}
 Let $A:\mathcal{S}(G)\rightarrow \mathcal{S}(G)$ be a pseudo-differential operator with symbol $\sigma:G\times \R^n\rightarrow\mathbb{C}$ satisfying the symbol inequalities
 \begin{equation}\label{ineq:symbIneqGrp::22}
    \abs{\pDiffGrpToOrd{x}{\beta} \pDiffToOrd{\xi}{\alpha} \sigma(x,\xi)}
      \leqslant C_{\alpha,\beta} (1+\abs{\xi})^{-\abs{\alpha}}(1+|\ln(1-x^{2})|)^{-|\beta|}.
\end{equation}
Then $A\in \Psi^{0}(G\times \R^n)$ {and its essential norm has the following estimate from below
\begin{equation}
    |\!|\!| A|\!|\!|_{\mathscr{B}(L^2(G))}:=\inf_{K\in\mathscr{C}(L^2(G))}\|A-K \|\geq d:=\limsup_{|\xi|\rightarrow\infty,\,|x|\rightarrow 1 }|\sigma(x,\xi)|.
\end{equation}}
Moreover, $A$ is {compact, $A\in\mathscr{C}(L^2(G))$,} if and only if $d=0.$
\end{theorem}
\begin{proof}
That $A\in \Psi^{0}_{1,0}(G\times \R^n),$ can be easily proved. Indeed, it is consequence of the inequality
\begin{equation}
  \abs{\pDiffGrpToOrd{x}{\beta} \pDiffToOrd{\xi}{\alpha} \sigma(x,\xi)}
     \leq    C_{\alpha,\beta} (1+\abs{\xi})^{-\abs{\alpha}}(1+|\ln(1-x^{2})|)^{-|\beta|}\lesssim (1+\abs{\xi})^{-\abs{\alpha}}.
\end{equation} Let $\tilde{A} : \SchwartzSpaceOn{\R^n} \to \SchwartzSpaceOn{\R^n}$ be the continuous linear operator such that the following diagram
\begin{eqnarray}
    \begin{tikzpicture}[every node/.style={midway}]
  \matrix[column sep={10em,between origins}, row sep={4em}] at (0,0) {
    \node(R) {$\SchwartzSpaceOn{\R^n}$}  ; & \node(S) {$\SchwartzSpaceOn{\R^n}$}; \\
    \node(R/I) {$\SchwartzSpaceOn{\group}$}; & \node (T) {$\SchwartzSpaceOn{\group}$};\\
  };
  \draw[<-] (R/I) -- (R) node[anchor=east]  {$\pull{t}$};
  \draw[->] (R) -- (S) node[anchor=south] {$\tilde{A}$};
  \draw[->] (S) -- (T) node[anchor=west] {$\pull{t}$};
  \draw[->] (R/I) -- (T) node[anchor=north] {$A$};
\end{tikzpicture}
\end{eqnarray}
commutes. Denote by $\tilde{a}(t,\xi)$ the symbol of $\tilde{A}.$ Then, in terms of $\tilde{a},$ \eqref{ineq:symbIneqGrp::22} is equivalent to the following symbol estimates
\begin{equation}\label{SG:zero}
  \abs{ \partial_{t}^\beta \pDiffToOrd{\xi}{\alpha} \tilde{\sigma}(t,\xi)}
     \leq    C_{\alpha,\beta} (1+\abs{\xi})^{-\abs{\alpha}}(1+|t|)^{-|\beta|} .
\end{equation}
Then, in view of the Gohberg lemma proved by Grushin \cite{Grushin}, for any compact operator {$\tilde{K}\in\mathscr{C}(L^2(\R^n))$},
\begin{equation}
    \Vert \tilde{A}-\tilde{K}  \Vert_{\mathscr{B}(L^2(\R^n))}\geq d:=\limsup_{|\xi|\rightarrow\infty,\,|t|\rightarrow \infty }|\tilde{\sigma}(t,\xi)|.
\end{equation}
Now, any compact operator {$K\in\mathscr{C}(L^2(G))$} is the pull-back by {$t_*:L^2(\R^n)\rightarrow L^2(G)$ }of compact operator {$\tilde{K}\in\mathscr{C}(L^2(\R^n))$}, that is $K=t_{*}(\tilde{K}),$ {and} then 
\begin{equation}
\Vert A-K  \Vert_{\mathscr{B}(L^2(G))}=    \Vert \tilde{A}-\tilde{K}  \Vert_{\mathscr{B}(L^2(\R^n))}\geq d:=\limsup_{|\xi|\rightarrow\infty,\,|t|\rightarrow \infty }|\tilde{\sigma}(t,\xi)|=\limsup_{|\xi|\rightarrow\infty,\,|x|\rightarrow 1 }|{\sigma}(x,t)|.
\end{equation}
The previous inequality proves the first part of Theorem \ref{Gohberg:Lemma} {by taking the infimum over all compact operators $K\in\mathscr{C}(L^2(G))$}. On the other hand, it was proved by Molahajloo in Theorem 1.4 of \cite{Molahajloo}, that $\tilde{A}$ is compact on $L^2(G),$ if and only if $d:=\limsup_{|\xi|\rightarrow\infty,\,|t|\rightarrow \infty }|\tilde{\sigma}(t,\xi)|=0.$ Since $A$ is compact on $L^2(G)$ if and only if $\tilde{A}$ is compact on $L^2(\R^n),$ then we have that $A$ is compact on $L^2(G)$ if and only if $d=0.$ The proof of Theorem \ref{Gohberg:Lemma} is complete.
\end{proof}

\subsection{The Atiyah-Singer-Fedosov index theorem on $G$ } In this section we compute the index  for the elliptic pseudo-differential operators with symbols in the sub-class
\begin{equation}\label{order:zero:shubin:class}
 \Sigma^{0}_{1,0}(G\times \R^n;\mathbb{C}^{\nu\times \nu})
\end{equation}
 of $S^{0}_{1,0}(G\times \R^n;\mathbb{C}^{\nu\times \nu})$ defined as follows.

For $\nu=1,$ $a^{G}\in \Sigma^{0}_{1,0}(G\times \R^n)= \Sigma^{0}_{1,0}(G\times \R^n;\mathbb{C}),$ if it  satisfies for every \(\alpha, \beta \in \N_{0}^{n}\) and every \((x,\xi) \in \group \times \R^n\) the symbolic estimates
\begin{equation}\label{ineq:symbIneqGrp:2}
    \abs{\pDiffGrpToOrd{x}{\beta} \pDiffToOrd{\xi}{\alpha} \symbToGrp{a}(x,\xi)}
    = \abs{(\pDiffToOrd{t}{\beta} \pDiffToOrd{\xi}{\alpha}a)(t(x),\xi)}
    \leqslant C_{\alpha,\beta} (1+|\ln(1-x^{2})|+\abs{\xi})^{-\rho\abs{\alpha}+\delta\abs{\beta}}.
\end{equation}
Conversely, if a smooth function \(a \in \CinftyOn{\group \times \R^{n}}\) satisfies the symbolic estimates \eqref{ineq:symbIneqGrp:2}, then we obtain for \(\symbToRn{a}(t,\xi) := a(x(t),\xi) \in \CinftyOn{\R^{n} \times \R^{n}}\) the inequalities
\begin{equation}\label{ineq:symbIneqGrp:2:2}
\abs{\pDiffToOrd{t}{\beta} \pDiffToOrd{\xi}{\alpha} \symbToRn{a}(t,\xi)} = \abs{(\pDiffGrpToOrd{x}{\beta} \pDiffToOrd{\xi}{\alpha} a)(x(t),\xi)} \leq C_{\alpha,\beta} (1+|t|+\abs{\xi})^{-\rho\abs{\alpha}+\delta\abs{\beta}}.
\end{equation}
So the class $\Sigma^{0}_{1,0}(\R^n\times \R^n)=\{a^{\R^n} \textnormal{ satisfying }\eqref{ineq:symbIneqGrp:2:2}\}$ is the Shubin class of order zero (see e.g. \cite[Page 459]{FischerRuzhanskyBook}). So, by following this nomenclature, we call to \eqref{order:zero:shubin:class} the Shubin class of order zero on $G.$

For $\nu \in \mathbb{N}$ with $\nu\geq 2,$ the class of matrix-symbols $S^{0}_{1,0}(G\times \R^n;\mathbb{C}^{\nu\times \nu})$ is the class of functions $a\in C^{\infty}(G\times \R^n,\mathbb{C}^{\nu\times \nu} ) $ whose entries are symbols in the class $S^{0}_{1,0}(G\times \R^n).$ So, to a matrix-symbol $a=(a_{ij})_{i,j=1}^{\nu}\in S^{0}_{1,0}(G\times \R^n;\mathbb{C}^{\nu\times \nu}) $ we associate the matrix-valued operator
$$ A\equiv \textnormal{Op}(a):=(\textnormal{Op}(a_{ij}))_{i,j=1}^{\nu}.  $$

We denote by 
\begin{equation}
    \PsiDOHor{0}{\rho}{\delta}{\group \times \R^n;\mathbb{C}^{\nu\times \nu}} = \{\sigma(x,\mathfrak{D}) : \sigma \in \symbClassOn{m}{\rho}{\delta}{\group \times \R^n;\mathbb{C}^{\nu\times \nu}}\}
\end{equation} the family of (global) matrix pseudo-differential operators on $\group$ with order $m$ and of $(\rho,\delta)$-type. It is clear that if
 $$ \Sigma^{0}_{1,0}(G\times \R^n;\mathbb{C}^{\nu\times \nu})=\{a=(a_{ij})_{ij=1}^{\nu}:a_{ij}\in \Sigma^{0}_{1,0}(G\times \R^n) \},$$
 then
\begin{equation}
    \Xi^{0}_{1,0}(G\times \R^n;\mathbb{C}^{\nu\times \nu})=\{\textnormal{Op}(a) : a=(a_{ij})_{ij=1}^{\nu}\in \Sigma^{0}_{1,0}(G\times \R^n;\mathbb{C}^{\nu\times\nu}) \}
\end{equation} is a subset of the operator class $\Psi^{0}_{1,0}(G\times \R^n;\mathbb{C}^{\nu\times \nu}).$

In general the class of operators $\Xi^{0}_{1,0}(G\times \R^n;\mathbb{C}^{\nu\times \nu})$ is a good sub-class of $\Psi^{0}_{1,0}(G\times \R^n;\mathbb{C}^{\nu\times \nu})$ if one is thinking in generalising several of the results that are available for pseudo-differential operators on compact manifolds. In particular, if as in this section one wants to compute an  Atiyah-Singer-Fedosov type formula. These formulas give topological data in terms of the symbols of the operators. Indeed, in 1974 B. Fedosov suggested in \cite{Fe74} the analytic index formula for an elliptic pseudo-differential operator, representing the index by the ``winding'' number of the symbol.  The following result is of Atiyah-Singer-Fedosov type. We prove, in particular, that the analytical index of an elliptic operator of the class $\Psi^0_{\rho,\delta}(G\times \R^n;\mathbb{C}^{\nu\times \nu})$ agrees with the ``winding'' number of its global symbol.
\begin{theorem}\label{S:Garding:Inequality}
 Let $A : \SchwartzSpaceOn{\group,\mathbb{C}^\nu} \to \SchwartzSpaceOn{\group,\mathbb{C}^\nu}$ be a pseudo-differential operator with symbol $a \in \Sigma^{0}_{1,0}(G\times \R^n;\mathbb{C}^{\nu\times\nu})$ such that for any $(x,\xi)$ outside of a ball $\B$ in $ G\times\R^n,$ $a(x,\xi)\in \textnormal{End}[\mathbb{C}^\nu]$ is invertible. Then $A$ is a Fredholm operator on $L^2(G,\mathbb{C}^{\nu})$
with index
 \begin{equation}\label{Index:A}
     \textnormal{ind}[A]=-\frac{(n-1)!}{(-2\pi i)^{n}(2n-1)!}\int\limits_{\partial{B}}\textnormal{Tr}[a^{-1}(x,\xi)da(x,\xi)]^{2n-1}
 \end{equation}where $G\times \R^n$ is positively oriented by the volume form $dx_1\wedge d\xi_1\wedge \cdots \wedge dx_n\wedge d\xi_n.$
\end{theorem}
\begin{proof} We start by observing that $A$ is bounded on $L^2(G)$ in view of the Calder\'on-Vaillancourt theorem in the form of Corollary \ref{Calderon-Vaillancourth}. Now, let us prove that $A$ is Fredholm on $L^2(G)$ and let us compute its index. Let $\tilde{A} : \SchwartzSpaceOn{\R^n} \to \SchwartzSpaceOn{\R^n}$ be the continuous linear operator such that the following diagram
\begin{eqnarray}
    \begin{tikzpicture}[every node/.style={midway}]
  \matrix[column sep={10em,between origins}, row sep={4em}] at (0,0) {
    \node(R) {$\SchwartzSpaceOn{\R^n}$}  ; & \node(S) {$\SchwartzSpaceOn{\R^n}$}; \\
    \node(R/I) {$\SchwartzSpaceOn{\group}$}; & \node (T) {$\SchwartzSpaceOn{\group}$};\\
  };
  \draw[<-] (R/I) -- (R) node[anchor=east]  {$\pull{t}$};
  \draw[->] (R) -- (S) node[anchor=south] {$\tilde{A}$};
  \draw[->] (S) -- (T) node[anchor=west] {$\pull{t}$};
  \draw[->] (R/I) -- (T) node[anchor=north] {$A$};
\end{tikzpicture}
\end{eqnarray}commutes. Denote by $\tilde{a}(t,\xi)$ the symbol of $\tilde{A}.$ In view of the Index theorem in $\R^n$ by H\"ormander (see \cite[Page 215]{Hormander1985III}), since $A : \SchwartzSpaceOn{\R^n,\mathbb{C}^\nu} \to \SchwartzSpaceOn{\R^n,\mathbb{C}^\nu}$ is a pseudo-differential operator  such that for any $(t,\xi)$ outside of a ball  $\tilde{B}$ of $ \R^n\times\R^n,$ $\tilde{a}(x,\xi)\in \textnormal{End}[\mathbb{C}^\nu]$ is invertible, then $\tilde{A}$ is a Fredholm operator on $L^2(\R^n,\mathbb{C}^{\nu})$
with {the} index
 \begin{equation}
     \textnormal{ind}[\tilde{A}]=-\frac{(n-1)!}{(-2\pi i)^{n}(2n-1)!}\int\limits_{\partial{\tilde{B}}}\textnormal{Tr}[\tilde{a}^{-1}(t,\xi)d\tilde{a}(t,\xi)]^{2n-1},
 \end{equation} 
 where $\R^n\times \R^n$ is oriented by the volume form $dt_1\wedge d\xi_1\wedge \cdots \wedge dt_n\wedge d\xi_n.$ The topology in $G$ induced by the mapping $t:G\rightarrow \R^n,$ implies that
 $$ B=t^{-1}[\tilde{B}],   $$ is a ball with respect to the distance $d_G(x,y)=|t(x)-t(y)|,$ and then
 $$ \partial B=t^{-1}[\partial\tilde{B}] . $$
 Because the index is a spectral invariant  under unitary transformations, we have that
 $$ \textnormal{ind}[\tilde{A}]=\textnormal{ind}[A].  $$
 Now, let us compute $ \textnormal{ind}[{A}]$ in terms of $a$. For this, note that
 \begin{align*}
     \tilde{a}^{-1}(t,\xi)d\tilde{a}(t,\xi) &=\tilde{a}^{-1}(t,\xi)\sum_{j=1}^n\frac{\partial \tilde{a}(t,\xi)}{\partial t_j}dt_j+\tilde{a}^{-1}(t,\xi)\sum_{j=1}^{n}\frac{\partial \tilde{a}(t,\xi)}{\partial \xi_j}d\xi_j\\
     &=a(x,\xi)^{-1}\sum_{j=1}^n(-(1-x_j^2))\frac{\partial a(x,\xi)}{\partial x_j}dt_j+{a}^{-1}(x,\xi)\sum_{j=1}^{n}\frac{\partial a(x,\xi)}{\partial \xi_j}d\xi_j\\
      &=a(x,\xi)^{-1}\sum_{j=1}^n(-(1-x_j^2))\frac{\partial a(x,\xi)}{\partial x_j}\times \frac{dx_j}{(-(1-x_j^2))}\\
      &\hspace{2cm}+{a}^{-1}(x,\xi)\sum_{j=1}^{n}\frac{\partial a(x,\xi)}{\partial \xi_j}d\xi_j\\
      &=a(x,\xi)^{-1}\sum_{j=1}^n\frac{\partial a(x,\xi)}{\partial x_j} {dx_j}+{a}^{-1}(x,\xi)\sum_{j=1}^{n}\frac{\partial a(x,\xi)}{\partial \xi_j}d\xi_j\\
      &=a(x,\xi)^{-1}da(x,\xi).
 \end{align*}In consequence, with $\tilde{C}=\partial \tilde{B}$ and $C=\partial B$ we have,
 \begin{equation*}
    \int\limits_{\tilde{C}}\textnormal{Tr}[\tilde{a}^{-1}(t,\xi)da(t,\xi)]^{2n-1}=\int\limits_{C}\textnormal{Tr}[a^{-1}(x,\xi)da(x,\xi)]^{2n-1}.
 \end{equation*} The proof is complete.
\end{proof}
\begin{remark}Note that in \eqref{Index:A}, $\omega=a^{-1}da$ is a  one form on $G\times \R^n$ with coefficients taking values in $\nu\times \nu$ matrices. If $n=\nu=1,$ the right-hand side of the equality \eqref{Index:A} is the winding number of $a$ considered as a mapping from $\partial B$ into $\mathbb{C}\setminus \{0\}.$ Then, in this case the index formula is
\begin{equation*}
    \textnormal{ind}(A)={-\frac{1}{2\pi i}\int\limits_{\partial B}\frac{da}a,}
\end{equation*}and we {re-obtain the index} formula by F. Noether \cite{NF}.
\end{remark}

\section{\bf Fredholm properties and $L^p$-boundedness  of $\Psi$DOs with non-classical symbols on \texorpdfstring{$\group = (-1,1)^n$}{group}}\label{Spectral:2}
In this section we analyse the Fredholmness and the $L^p$-boundeness of pseudo-differential operators by using result of the same type for Fourier multipliers on the group.
\subsection{Boundedness of pseudo-differential operators in the Bessel potential spaces}

In the present subsection we will expose conditions on a symbol $a(x,\xi)$ which ensure the boundedness of the corresponding pseudo-differential operator $a(x,\fD)$ in \eqref{e2.3} in the Bessel potential spaces
 \begin{eqnarray}\label{e6.8}
a(x,\fD)\;:\;\bH_p^s(G)\rightarrow\bH_p^{s-r}
     (G)
 \end{eqnarray}

We start with convolution operators $a_{\R^n}(D)$ and $a_\group(\fD)$, which are pseudo-differential, but independent of the variable $x$.

\begin{definition}[Polytopes]
Consider polytope domain $\Omega$ in $\R^n$ which represents an open domain, intersection of finite number of affine transformed half spaces
 \[
\bigcap_{j=1}^N\left\{\xi\in\R^n\;:\;\langle{\bf V}^j, \xi-\eta^j\rangle>0,\;\; N\geqslant n+1\right\}\not=\emptyset,\quad {\bf V}^1,\ldots,{\bf V}^N,\eta^1,\ldots\eta^n\in\R^n.
 \]
Polytope domains include cones open to infinity and having empty intersection with a sufficiently large polytope domain, containing 0. Examples of polytopes in $\mathbb{R}^3$ are given in Figure \ref{Pol}.
\end{definition}
\begin{definition}[Piecewise constant functions]
We call a function $a(\xi)$, $\xi\in\R^n$, piecewise-constant if $\R^n$ is divided into a finite number of polytops and $a(\xi)$ is constant on each of these polytops. ${\bf P}{\bf C}(\R^n)$ denotes the algebra of piecewise-constant functions.
\end{definition}
\begin{figure}[h]
\includegraphics[width=12cm]{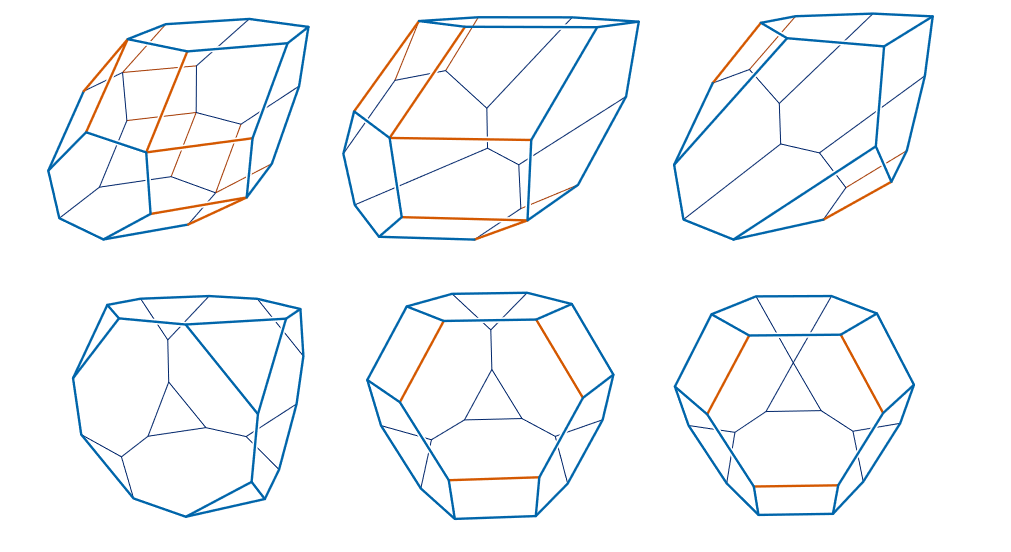}\\
\caption{Examples of polytopes in $\mathbb{R}^3,$ see \cite{Pol}.}
 \label{Pol}
\centering
\end{figure} 

 %
 \begin{lemma}\label{l6.5} Any piecewise constant function  belongs to the multiplier class\linebreak ${\bf P}{\bf C}(\R^n)\subset\widetilde{\fM}_p(\R^n) \subset\fM_p(\R^n)$ for all $1<p<\infty$.
 \end{lemma}
 \noindent
\begin{proof}
For the proof we consider the following steps:
\begin{itemize}
    \item Step 1: We define a polysingular integral operator ${\bf S}_{G,k}$. We recall some identities between the operator ${\bf S}_{G,k} $ and the partial Hilbert transforms.
    \item Step 2. We define the class of affine transformations on $\mathbb{R}^n.$ These transformations are applied to the operator and an identity involving the symbol $a(\xi)$ with the symbol of partial Hilbert transforms allow us to conclude the proof. 
\end{itemize}
So we continue the proof under this strategy. 
\begin{itemize}
    \item Step 1:  The Cauchy polysingular integral operators
 \begin{eqnarray}\label{e6.9}
{\bf S}_{G,k}\varphi(x):=\frac1{\pi}\int\limits_{-1}^1
     \frac{\varphi(x_1,\ldots,x_{k-1},y_k,x_{k+1},\ldots,x_n)}{
     \dst\frac12\ln\dst\frac{1+x_k\circ(-y_k)}{1-x_k\circ (-y_k)}}\frac{dy_k}{1 -y^2_k},\\ \quad x=(x_1,\ldots,,x_n)^\top\in G,\nonumber
 \end{eqnarray}
are bounded in $L^p(G)$ for all $1<p<\infty$ and all $k\in \mathbb{N}$ because their counterpart on the Euclidean space $\R^n$
 \begin{eqnarray*}
{\bf S}_{\R,k}\psi(t):=x_*{\bf S}_{G,k}t_*\psi(t)
     =\frac1{\pi}\int\limits_{-\infty}^\infty\frac{\psi(t_1,\ldots,t_{k-1},\tau_k,t_{k+1},
     \ldots,t_n)d\tau_k}{\tau_k-t_k},\\ t=(t_1,\ldots,,t_n)^\top\in\R,
 \end{eqnarray*}
are partial Hilbert transforms (which are polysingular operators) and are bounded in $L^p(\R^n)$ (see, e.g. \cite{Du76}). These operators are, obviously, of  convolution type 
 \[
{\bf S}_{\R,k}=W^0_{\R,\cS_k},\qquad {\bf S}_{G,k}
     =W^0_{G,\cS_k},\qquad k=1,\ldots,n,
 \]
and the symbol of both of them is (cf. \cite[Lemma 1.35, p.23]{Du79} $\cS_k(\xi)=-{\rm sign}\,\xi_k$, $\xi=(\xi_1,\ldots,\xi_n)^\top\in\R^n$.

Therefore, the operators
 \[
\frac12\left[I\mp{\bf S}_{{G,k}}\right]
     =W^0_{G,\cH^k_\pm},\qquad \cH^k_\pm(\xi):=\frac12(1\pm{\rm sign}\,\xi_k)\quad k=1,\ldots,n
 \]
are bounded in the space $L^p(G)$  and $\cH^k_\pm\in\fM_p(\R^n)$ are multipliers for all $1<p<\infty$ and all $k=1,\ldots,n$. Note, that,  $\cH_k(\xi)$ are the characteristic functions of the coordinate half spaces $\Pi_\pm^k:=\{\xi\in\R^n\;:\;\pm\xi_k>0\}$, $k=1,\ldots,n$.

\item Step 2. 
On the other hand, for any vector $\eta\in\R^n$, any affine transformation
 \[
\sigma\in{\bf A}{\bf T}(\R^n):=\{\sigma\;:\;\R^n\to\R^n\;:\; \sigma^\top
     =\sigma^{-1}\}
 \]
and any multiplier $a\in\fM_p(\R^n)$  we have the inclusion $\sigma_*a, V_\eta a\in\fM_p(\R^n)$, where $\sigma_*u(\xi):=u(\sigma\xi)$ and $V_\eta a(\xi)=a(\xi-\eta)$.

Indeed, note first that
 \[
\sigma_\ast\;:\;{H}_p^s(\R^n)\to{H}_p^s(\R^n), \qquad
     e^{\pm ict}I\;:\;{H}_p^s(\R^n)\to{H}_p^s(\R^n)
 \]
are isometric isomorphisms of the spaces. Then the claimed inclusions $\sigma_*a, V_\eta a\in\fM_p(\R^n)$ follow from the equalities:
 \begin{eqnarray*}
\sigma_\ast a(D)\sigma^\top_\ast\varphi(t)&\hskip-3mm=&\hskip-3mm
     \frac1{(2\pi)^n}\int\limits_{\R^n}e^{-i\sigma t\cdot\xi}a(\xi)d\xi\int\limits_{\R^n}
     e^{i\xi\cdot y}\sigma_\ast^\top\varphi(y)dy=(\sigma_\ast a)(D)\varphi(t),\\
e^{-i\eta\cdot\xi}a_{\R^n}(D)e^{i\eta\cdot y}\varphi(x)&\hskip-3mm=&\hskip-3mm
     \int\limits_{\R^n}e^{-it\cdot(\xi+\eta)}a(\xi)d\xi\int\limits_{\R^n}e^{i(\xi-\eta)\cdot y}
     \varphi(y)dy=a_{\R^n}(D-\eta I).
 \end{eqnarray*}

Therefore, the characteristic function $\cH_{\Pi(\sigma,\eta)}$ of any half space $\Pi(\sigma,\eta)$, situated from one side of any hyperplane $\sigma\Pi_+^1-\eta$, $\sigma\in {\bf A}{\bf T}(\R^n)$, $\eta\in\R^n$, is a multiplier $\cH_{\Pi(\sigma,\eta)}\in\fM_p(\R^n)$.

Any piecewise constant function $a\in {\bf P}{\bf C}(\R^n)$ is represented in the form
 \[
a(\xi)=\sum_{k=1}^Nd_k\prod_{j=1}^M \cH^1_+(\sigma_{jk}\xi-\eta_{jk}),\qquad d_k\in\mathbb{C}, \quad \eta_{jk}\in\R^n,\quad j,k=1,\ldots,n,
 \]
where $\sigma_{jk}\in{\bf A}{\bf T}(\R^n)$,  $j=1,\ldots,M$, $k=1,\ldots,N$. Then $a\in\fM_p(\R^n)$, because  the multiplier class $\fM_p(\R^n)$ is an algebra $W^0_{G,g}W^0_{G,h} =W^0_{G,gh}$.
\end{itemize} The proof of Lemma \eqref{l6.5} is complete.
\end{proof}
 %
 
Now we analyse the boundedness of pseudo-differential operators using the norm of $L^p$-Fourier multipliers. 
 \begin{theorem}\label{t6.6}
 Let $1<p<\infty$ and let $m\in \mathbb{N},$ and  assume that a symbol $a(x,\xi)$ satisfies the following
 hypothesis:
 \[
 {\bf A.}\hskip30mm a_{(\beta+\gamma)}\in C(\R^n,\fM_p(\R^n)), \quad
      \beta,\gamma\in\N_0^n,\quad |\beta|\leqslant1,\quad|\gamma|\leq m;
 \]
for any permutation $\kappa$ of the variables $y_1,\ldots, y_n$ and for any $k$, $0<k\leqslant n$;
\begin{equation}\label{e6.11}
{\bf B.}\quad M_{\kappa,\gamma}:=\int\limits_{\R^k}\left\|\frac{
    \partial^k a^\kappa_{(\gamma)}(y_1,\ldots,y_k,0\ldots,0,\cdot)}{\partial
    y_1\ldots\partial y_k}\right\|_{ \fM_p(\R^n) }dy_1 \ldots dy_k<\infty,
\end{equation}
where $$ a^\kappa_{(\gamma)}(y,\xi):=a_{(\gamma)}(\kappa(y),\xi)$$ and $|\gamma|\leq m$.  Then the pseudo-differential operators
 \begin{eqnarray}\label{e6.12}
a_{\R^n}(x,D)\;:\;{H}_p^s(\R^n)\to {H}_p^s(\R^n),\\
\label{e6.13}
a_\group(x,\fD)\;:\;{H}_p^s(G)\to{H}_p^s(G)
 \end{eqnarray}
are bounded provided that
$$ -m-1+\frac{1}{p}<s<m+\frac{1}{p},$$ and the corresponding operators norms satisfy the estimate
 \begin{eqnarray}\label{e6.14}
\|a_{\R^n}(x,D)\|_{ \mathscr{B}({H}_p^s(\R^n)) }=\|a_\group(x,D)
    \|_{ \mathscr{B}({H}_p^s(G)) }\leqslant C_{s,p} \sup_\kappa\sum\limits_{\gamma|\leqslant m}M_{\kappa,\gamma},
 \end{eqnarray}
where the constant $C_{s,p}<\infty$ are depending  only on the parameters $s$, $p$ and $\mathscr{B}(\fB)$ denotes the algebra of bounded operators in the Banach space $\fB$.
\end{theorem}
\begin{proof} The continuity of the operator $a_{\R^n}(x,D)$ in \eqref{e6.12} is proved by E. Shargorodsky in \cite[Theorem 5.1]{Sh97}.

The continuity of the operator $a_{\group}(x,\fD)$ in \eqref{e6.13} is proved similarly to Theorem \ref{MainTheorem} and Theorem \ref{S:Garding:Inequality} above, based on the continuity of the operator $a_{\R^n}(x,D)$ in \eqref{e6.12}.
\end{proof}
\begin{definition}[The class  ${\bf P}{\bf C}_p(\R^n)$]
Let ${\bf P}{\bf C}_p(\R^n)$ denote the subalgebra of $\fM_p(\R^n)$, obtained by closing the algebra ${\bf P}{\bf C}(\R^n)$ of piecewise constant functions on polytope domains in the norm of the multiplier algebra $\fM_p(\R^n)$.
\end{definition}

From Lemma \ref{l6.5} and Theorem \ref{t6.6}  it  follows that we have the following property.
 %
 \begin{corollary}\label{c6.7}
 Let $1<p<\infty$, $m\in \mathbb{N}$ and assume that  $a=a(x,\xi)$ satisfies the following symbol conditions
\begin{equation}\label{e6.15}
a_{(\beta+\gamma)}\in C(\R^n,{\bf P}{\bf C}_p(\R^n)), \quad
      \beta,\gamma\in\N_0^n,\quad |\beta|\leqslant1,\quad|\gamma|\leq m;
\end{equation}
If the condition \eqref{e6.11} holds, then the pseudo-differential operator $a_{\R^n}(t,D)$ in \eqref{e6.14} and the pseudo-differential operator $a_G(x,\fD)$ in \eqref{e6.15} are bounded and the inequality \eqref{e6.14} holds.
\end{corollary}
\begin{remark}
By applying the well known inequality
 \begin{eqnarray}\label{e6.15}
\|a\|_{\fM_p(\R^n) }:=\|W^0_{a,G}\|_{L^p(G)}=\|W^0_{a,\R^n}\|_{L^p(\R^n) }\nonumber\\
     \geqslant\|a\|_{\fM_2(\R^n)}=\|a\|_{L^\infty(\R^n)  }
 \end{eqnarray}
(cf. \cite{Hr60}), it can easily be checked, that any function $a(\xi)$ from ${\bf P}{\bf C}_p(\R^n)$ have radial limits at any point including the infinity $\xi\in\R^n_\bullet$ (the spherical compactification of $\R^n$ at the infinity), where limits are taken along all beams emerging from $\xi\in\R^n_\bullet$:
 \begin{eqnarray}\label{e6.16}
 a(\xi,\omega):=\lim_{\lambda\to0} a(\xi+\lambda\omega), \qquad a(\infty,\omega):=\lim_{\lambda\to\infty} a(\lambda\omega)\\
 \mbox{for all}\quad \xi\in\R^n,\quad \omega\in\B_1(0),\quad|\omega|=1.\nonumber
 \end{eqnarray}
where $\B_1(0)$ is the unit sphere centered at $0$.
\end{remark}

\subsection{Fredholm properties of $\Psi$DOs} Here we prove a theorem on Fredholm properties of $\Psi$DOs with non-classical symbols on the group $G$. First we investigate a particular case of convolution operators. We denote by $\R^n_\bullet$ the  spherical compactification of $\R^n$ at the infinity.
\begin{theorem}\label{t6.8}
Let $a(\xi)$ be a complex valued matrix symbol $a\in{\bf P}{\bf C}^{m\times m}_p(\R^n)$.
The convolution operators
 \begin{eqnarray}\label{e6.17}
a_{\R^n}(D)=W^0_{a,\R^n}\;:\;{H}_p^s(\R^n)\to {H}_p^s(\R^n),\\
\label{e6.18}
a_\group(\fD)=W^0_{a,G}\;:\;{H}_p^s(G)\to{H}_p^s(G)
 \end{eqnarray}
are Fredholm if and only if the symbol is elliptic
 \[
  \inf_{(\xi,\omega)\in\R^n_\bullet\times\B_1(0)}|a(\xi,\omega)|>0.
 \]
Moreover, if the ellipticity holds, the operators in \eqref{e6.17} and in \eqref{e6.18} are invertible for all $p\in(1,\infty)$ and all $s\in\R$ and the inverse operators read, respectively, $W^0_{a^{-1},\R^n}$ and $W^0_{a^{-1},G}$.
\end{theorem}
\begin{proof} We divide the proof in the following steps.

\begin{itemize}
    \item Step 1. We apply the Gohberg-Krupnik's localization method.
    \item Step 2. We introduce a class of localising operators $\Delta_{(\xi_0,\omega_0)}.$ 
    \item We locally  approximate the symbol $a(\xi)$ with the elements of the class $\Delta_{(\xi_0,\omega_0)}$ and we conclude the proof. 
\end{itemize}
 Let us start with the  Gohberg-Krupnik's localization method.
 
\begin{itemize}
    \item Step 1.
We will apply Gohberg-Krupnik's localization method, described in \cite[Ch. 5]{GK79}, \cite[\S\,1.$7^0$]{Du79}, \cite{SC86} and \cite{BKS88}. Let us consider the set $\R^n_\bullet\times\B_1(0)$, where $\R^n_\bullet:=\R^n\cup\{\infty\}$ denotes the one point compactifications of the space $\R^n$. We cover the set $\R^n_\bullet\times\B_1(0)$ by  polytope cones, defined as follows. Neighbourhoods of a point $(\xi_0,\omega_0)$, $\xi_0\not=\infty$, $\omega_0\in\B_1(0)$, is a part of polytope cone with the vertex at $\xi_0$, inside a small polytope domain containing $\xi_0$ and containing the point $\xi_0+\varepsilon\omega_0$ for some small $\varepsilon>0$. Neighbourhoods of a point $(\infty,\omega_0)$, $\omega_0\in\B_1(0)$, is a part of polytope cone with the vertex at $0$, open to infinity, outside a big polytope neighbourhood of $0$ and containing the point $R\omega_0$ for some large $R>0$.

\item  The localizing class $\Delta_{(\xi_0,\omega_0)}$, $(\xi_0,\omega_0)\in\R^n_\bullet\times\B_1(0)$, comprizes operators $W^0_{h,G}$, with symbols $h(\xi)$ which are characteristic functions of all kind of non-empty polytope neighbourhoods of $(\xi_0,\omega_0)$. It is clear, that the system of localizing classes  $\left\{\Delta_{(\xi_0,\omega_0)}\right\}_{(\xi_0,\omega_0)\in\R^n_\bullet\times\B_1(0)}$ are covering in the algebra of linear bounded operators $$ \mathscr{B}({H}_p^s(\R^n)).$$  Indeed, from any collection of operators $\left\{W^0_{h_{(\xi_0,\omega_0)},G}\right\}_{(\xi_0,\omega_0)\in\R^n_\bullet\times\B_1(0)}$ we can select a finite number of operators $\left\{W^0_{h_k,G}\right\}_{k=1}^N$ such that the sum
\[
\sum_{k=1}^nW^0_{h_k,G}=W^0_{h_0,G},\qquad h_0(\xi):=\sum_{k=1}^nh_k(\xi)
\]
is invertible ($h_0\in{\bf P}{\bf C}^{m\times m}_p(\R^n)$ is elliptic) and the inverse reads $W^0_{h^{-1}_0,G}$, because $h^{-1}_0\in{\bf P}{\bf C}^{m\times m}_p(\R^n)$.

Obviously, operators from $\Delta_{\xi_0,\omega_0}$ commute with convolutions $W^0_{b,G}$, $b\in{\bf P}{\bf C}^{m\times m}_p(\R^n)$, because they are all convolutions.

Using the well known H\"ormander's inequality (cf. \cite{Hr60})
 \begin{eqnarray}\label{e6.19}
\|a\|_{\fM_p(\R^n)}=\|a\|_{\fM_{p'}(\R^n)}=\|W^0_{a,G}\|_{L^p(G)}
     =\|W^0_{a,\R^n}\|_{L^p(\R^n)}\nonumber\\
\leqslant\left[\sup_{\xi\in\R^n}|a(\xi)|\right]^\theta\|a\|_{\fM_r(\R^n)}^{1-\theta},\\
    p\not=2,\qquad p\in(r,r'), \qquad p'=\frac{p}{p-1},\qquad\theta=\frac2p\frac{|r-p|}{|r-2|},\nonumber
 \end{eqnarray}
 \item 
we prove the local equivalence
 \[
W^0_{a,G}\overset{\Delta_{(\xi_0,\omega_0)}}{\sim} W^0_{a(\xi_0,\omega_0),G}
     =a(\xi_0,\omega_0)I, \qquad (\xi_0,\omega_0)\in\R^n_\bullet\times\B_1(0),
 \]
(see \eqref{e6.16} for $a(\xi_0,\omega_0)$) which means that
 \begin{eqnarray*}
 \inf_{W^0_{h,G}\in\Delta_{(\xi_0,\omega_0)}}\|W^0_{h,G}[W^0_{a,G}-a(\xi_0,\omega_0)I]
      \big|L^p(G)\|\\
 =\inf_{W^0_{h,G}\in\Delta_{(\xi_0,\omega_0)}}\|W^0_{h[a-a(\xi_0,\omega_0)]}
      \big|L^p(G)\|=0.
 \end{eqnarray*}

 According to the main theorem on localization operators, $W^0_{a,G}$ is Fredholm if and only if the local representatives (which are in our case multiplication by constants $a(\xi_0,\omega_0)I$) are locally invertible for all $(\xi_0,\omega_0)\in\R^n_\bullet\times\B_1(0)$: There exists operator $R_{(\xi_0,\omega_0)}$ and an element of the localazin class $W^0_{h,G}\in\Delta_{(\xi_0,\omega_0)}$ such that the equalities hold:
 \[
W^0_{h,G}a(\xi_0,\omega_0)R_{(\xi_0,\omega_0)}=R_{(\xi_0,\omega_0)}a(\xi_0,\omega_0)W^0_{h,G}
=W^0_{h,G}\qquad \forall\; (\xi_0,\omega_0)\in\R^n_\bullet\times\B_1(0).
 \]
But the local invertibility of the multiplication by a constant $a(\xi_0,\omega_0)I$ coincides with its global invertibility and coincides with the ellipticity condition at the point $a(\xi_0,\omega_0)\not=0$. Thus, ellipticity means Fredholmness, but since for an elliptic symbol $a(\xi)$ the inverse is also a multiplier $a^{-1}\in{\bf P}{\bf C}^{m\times m}_p(\R^n)$, the inverse to $W^0_{a,G}$ is $W^0_{a^{-1},G}$. A similar property holds for the operator $W^0_{a,\R^n}$.
\end{itemize} 
\end{proof}
\begin{definition}[A class of homogeneous of order zero symbols]
For a function $a(x,\xi)$ satisfying conditions \eqref{e6.15}, we define the following limit functions
 \begin{eqnarray}\label{e6.20}
a_\infty(x,\xi):=\lim_{\lambda\to\infty}a(x,\lambda\xi)
 \end{eqnarray}
which are homogeneous of order $0$ in $\xi$: $a_\infty(x,\theta\xi)=a_\infty(x,\xi)$ for all $\theta>0$, $x\in\R^n$, $\xi\in\R^n$.
\end{definition}

Prior we formulate and prove the main theorem of the present section, we expose the result about local operators, due to V.N. Semenyuta and A.V. Kozak and exposed with the proof in \cite[Proposition 5.7]{BKS88} and in \cite[Definition 2.5, Theorem 2.5]{SC86}.

\begin{definition}[Local type operators] For any $s\in \mathbb{R},$  let $\mathscr{C}(H^s_p (G))$ be the family of compact operators on $H^s_p (G).$ An
operator $A\in\mathscr{B}({H}^s_p (G))$ is of local type if for all $v_1,v_2\in C^\infty(G)$ such that ${\rm supp}\,v_1\cap{\rm supp}\,v_2=\emptyset,$ the localisation operator $$ v_1Av_2I$$ is a compact operator on  $H^s_p (G),$ that is   $v_1Av_2I\in\mathscr{C}(H^s_p (G)).$ The class of local type operators on $\mathscr{B}(H^s_p (G)$ will be denoted by $\mathscr{B}({LH}^s_p (G)).$
\end{definition}

In the following proposition the class of local type operators is characterised by a commutator criterion.
%
\begin{proposition}\label{p6.9}
Operator $A\in\mathscr{B}({LH}^s_p(G))$ is of local type if and only if
the commutator $[A,bI]:=AbI-bA$ is compact $[A,bI]\in\mathscr{C}({LH}^s_p (G,d\mu_G))$ for all smoooth functions $b\in C^\infty(\dG)$, where $\dG$ is the one point compactifications of the Lie group $G,$ namely, $\dG=G\cup\{\infty\}$.
\end{proposition}
%
\begin{definition}\label{d6.10}
Let $C^0_p(G,{\bf P}{\bf C}_p(\R^n))$ denote the class of symbols of $\Psi$DOs which are closure (in the norm of $\mathscr{B}(L^p(G)$) of operators
 \begin{eqnarray}\label{e6.21}
\sum_{k=1}^na_k(x)W^0_{g_k,G},\qquad a_k\in C(\dG), \quad g_k\in{\bf P}{\bf C}_p(\R^n),
      k=1,2,\ldots,N,
 \end{eqnarray}
when $n=1,2,\ldots$ is not fixed.
\end{definition}
\begin{remark}
Note that, due to Proposition \ref{p6.9}, the commutator of the operators $a(x,\fD),\; b(x,\fD)$ with symbols from the class $C^0_p(G,{\bf P}{\bf C}_p(\R^n))$, is compact
 \begin{eqnarray}\label{e6.22}
a(x,\fD)b(x,\fD)-b(x,\fD)a(x,\fD)\in\mathscr{C}(L^p(G)).
 \end{eqnarray}
Indeed, for this we only need to check that the commutator $[aW^0_{b,G},W^0_{b,G}aI]:=aW^0_{b,G}-W^0_{b,G}aI$, $a\in\Delta_{x_0}$, $b\in{\bf P}{\bf C}_p(\R^n)$ is compact. It is easy to ascertain that $W^0_{b,G}$ are operators of local type: $v_1W^0_{b,G}v_2I\in\mathscr{C}({LH}^s_p (G,d\mu_G))$ for any functions $v_1$ and $v_2$ with disjoint supports, because its kernel $k(x,y)$ is $C^\infty$-smooth and uniformly bounded. Then, due to Proposition \ref{p6.9}, the commutator is also compact $[aW^0_{b,G},W^0_{b,G}aI]\in\mathscr{C}({LH}^s_p (G,d\mu_G))$.
\end{remark}

%
\begin{theorem}\label{t6.11}
Let $1<p<\infty$ and let a matrix symbol $a=a(x,\xi)\;:\;\dG\times\R^n_\bullet\to \C^{m\times m}$ belong to the class $C^0_p(G,{\bf P}{\bf C}_p(\R^n))$.

The $\Psi$DO
 \begin{eqnarray}\label{e6.23}
a_\group(x,\fD)\;:\; L^p (G)\to L^p (G)
 \end{eqnarray}
is Fredholm if and only if the symbol is elliptic
 \[
  \inf_{(x,\xi,\omega)\in\dG\times\R^n_\bullet\times\B_1(0)}|a(x,\xi,\omega)|>0.
 \]
\end{theorem}
\begin{proof} We will apply again Gohberg-Krupnik's localization method (see Theorem \ref{t6.8}). Let us define localizing class $\Delta_{(x_0,\xi_0,\omega_0)}$ at the point $(x_0,\xi_0,\omega_0)\in\dG\times\R^n_\bullet\times\B_1(0)$, consisting of operatirs $gW^0_{h,G}$, where $|g(x)|\leqslant1$ is a smooth function, equal 1 in some neighbourhood of $x_0\in\dG$ and $h(\xi)$ is described in the proof of Theorem \ref{t6.8}.

We can prove invertibility of the factor-class $[a_G(x,\fD)]$ in the Kalkin factor-algebra of linear bounded operators $\mathscr{B}(L^p(G))/\mathscr{C}(L^p(G))$, because invertibility of the class in the Kalkin algebra is equivalent to the Fredholmness of the operator in the algebra of linear bounded operators $\mathscr{B}(L^p(G))$.

Since operators from localizing classes $\Delta_{x_0,\xi_0,\omega_0}$ have the form \eqref{e6.23}, their commutators with $a_G(x,\fD)$ are compact (cf. \eqref{e6.22}), i.e. the corresponding classes commute in the Kalkin algebra.

It is clear, that the system of localizing classes  $\left\{\Delta_{(x_0,\xi_0,\omega_0)}\right\}_{(x_0,\xi_0,\omega_0)\in\dG\times\R^n_\bullet\times\B_1(0)}$
is covering in the Kalkin algebra $\mathscr{B}(L^p(G))/ \mathscr{C}(L^p(G))$. We prove easily the local equivalences:
 \begin{eqnarray*}
[a_G(x,\fD)]\overset{\Delta_{(x_0,\xi_0,\omega_0)}}{\sim} [W^0_{a(x_0,\xi_0,\omega_0),G}]
     =[a(x_0,\xi_0,\omega_0)I], \qquad (x_0,\xi_0,\omega_0)\in\dG\times\R^n_\bullet\times\B_1(0).
 \end{eqnarray*}

According the main theorem on localization operator $a_G(x,\fD)$ is Fredholm if and only if the local representatives are locally invertible for all $x_0\in\dG$. But the local incvertibility of the class of multiplication by a constant $[a(x_0,\xi_0,\omega_0)I]$ coincides with its global invertibility and coincides with the ellipticity condition at the point $a(x_0,\xi_0,\omega_0)\not=0$. Thus, ellipticity means Fredholmness.
\end{proof}
%
\begin{remark}\label{r6.12} Observe that similarly to $C^0_p(\dG,{\bf P}{\bf C}_p(\R^n)),$ one  defines the class of symbols $C^0_p(\R^n_\bullet,{\bf P}{\bf C}_p(\R^n))$ by taking coefficients $a_k\in C(\R^n_\bullet)$ in \eqref{e6.21} and taking the approximation in the space $\mathscr{B}(L^p(\R^n))$.

A  similar result  to Theorem \ref{t6.11} holds for the operator
 \begin{eqnarray}\label{e6.24}
a_{\R^n}(x,D)\;:\;L^p(\R^n)\to L^p(\R^n)
 \end{eqnarray}
with a symbol from the class $C^0_p(\R^n_\bullet,{\bf P}{\bf C}_p(\R^n))$.

For the class of symbols $C^{s,r}_p(\dG,{\bf P}{\bf C}^r_p(\R^n))$ (and $C^{s,r}_p(\R^n_\bullet,{\bf P}{\bf C}^r_p(\R^n))$) obtained by closing the set of $\Psi$DOs \eqref{e6.21} with smooth coefficients $a_k\in C^\infty$ and symbols $(1+\xi^2)^{-r}g_k\in{\bf P}{\bf C}_p(\R^n)$ in the norm of $\mathscr{B}(\bH^s_p(G,d\mu_G),\bH^{s-r}_p(G,d\mu_G))$ (in the norm of $\mathscr{B}({H}_p^s(\R^n),\bH^{s-r}_p(\R^n))$), a theorem similar to Theorem \ref{t6.11} can be proved for the operator
 \begin{eqnarray*}
a_\group(x,\fD)\;:\;{H}_p^s(G)\to\bH^{s-r}_p(G)\\
({\rm for}\quad a_{\R^n}(x,D)\;:\;{H}_p^s(\R^n)\to \bH^{s-r}_p(\R^n)).
 \end{eqnarray*}
For the proof we first lift the operators to the space setting $ L^p \to L^p $ with the help of Bessel potentials and then apply Theorem \ref{t6.11}.
\end{remark}
\bibliographystyle{amsplain}

\end{document}